%% file: Optimization_of_Perron_eigenvector_and_eigenvalue.tex
\documentclass[final,times,1p]{elsarticle}

\usepackage{ifthen}
\newboolean{arxiv}
\setboolean{arxiv}{false}

\usepackage{hyperref}

\usepackage{amsthm}
\usepackage{amsmath}
\usepackage{amssymb}

\usepackage{graphicx}

\usepackage{algorithm}

\usepackage{tikz}
\usetikzlibrary{snakes,shapes,arrows} 
\theoremstyle{plain}
\newtheorem{thm}{Theorem}
\newtheorem{cor}{Corollary}
\newtheorem{prop}{Proposition}
\newtheorem{countex}{Counter example} 
\newtheorem{lemma}{Lemma} 

\theoremstyle{definition}
\newtheorem{defn}{Definition}
\newtheorem{remark}{Remark}
\newtheorem{heuristic}{Heuristic}

\newcommand{\norm}[1]{\lVert #1 \rVert}
\newcommand{\diag}{\mathrm{diag}}

\newcommand{\abs}[1]{\lvert #1 \rvert}

\newcommand{\bigO}{\text{\large O}}

\newcommand{\D}{\mathrm{D}}
\newcommand{\C}{\mathcal{C}}

\newcommand{\NEW}[1]{{\em #1}}

\title{Perron vector optimization applied to search engines}

\ifthenelse{\boolean{arxiv}}
{
\author{Olivier Fercoq \footnotemark[1] \footnotemark[2]
}
}
{
\author{Olivier Fercoq\corref{1}\corref{2}}%
\cortext[1]{The doctoral work of the author is supported by Orange Labs through the research contract CRE~3795 with INRIA}%
\cortext[2]{INRIA Saclay and CMAP Ecole Polytechnique \ead{olivier.fercoq@inria.fr}}
}

\begin{document}

\ifthenelse{\boolean{arxiv}}
{
 \footnotetext[1]{INRIA Saclay and CMAP Ecole Polytechnique\\
 {\tt\small olivier.fercoq@inria.fr}}%
 \footnotetext[2]{The doctoral work of the author is supported by Orange Labs through the research contract CRE~3795 with INRIA}%

\maketitle
}
{}

\begin{abstract}
In the last years, Google's PageRank optimization problems have been extensively studied. In that case, the ranking is given by the invariant measure of a stochastic matrix. In this paper, we consider the more general situation in which the ranking is determined by the Perron eigenvector of a nonnegative, but not necessarily stochastic, matrix, in order to cover Kleinberg's HITS algorithm. We also give some results for Tomlin's HOTS algorithm. The problem consists then in finding an optimal outlink strategy subject to design constraints and for a given search engine.

We study the relaxed versions of these problems, which means that we should accept weighted hyperlinks. We provide an efficient algorithm for the computation of the matrix of partial derivatives of the criterion, that uses the low rank property of this matrix. We give a scalable algorithm that couples gradient and power iterations and gives a local minimum of the Perron vector optimization problem. We prove convergence by considering it as an approximate gradient method.

We then show that optimal linkage stategies of HITS and HOTS optimization problems verify a threshold property. We report numerical results on fragments of the real web graph for these search engine optimization problems.
\end{abstract}

\ifthenelse{\boolean{arxiv}}
{}
{
\maketitle
}

\section{Introduction}

\paragraph{Motivation}

Internet search engines use a variety of algorithms to sort web pages based on their text
content or
on the hyperlink structure of the web.
In this paper, we focus on algorithms that use the latter hyperlink structure, 
called link-based algorithms. The basic notion for all these algorithms is the web graph,
which is a digraph with a node for each web page and an arc between pages~$i$ 
and~$j$ if there is a~hyperlink from page~$i$ to page~$j$. Famous link-based algorithms are PageRank~\cite{Brin-Anatomy}, HITS~\cite{Kleinberg-HITS}, SALSA~\cite{Lempel-Salsa} and HOTS~\cite{Tomlin-HOTS}.
See also~\cite{LangvilleMeyer-SurveyEigenvectorMethods, LanMey-Beyond} for surveys on these algorithms.
The main problem of this paper is the optimization of the ranking of a given web site.
It consists in finding an optimal outlink
strategy maximizing a given ranking subject to design constraints.
%subject to design constraints and for a given search engine.

One of the main ranking methods relies on the PageRank introduced by Brin and Page~\cite{Brin-Anatomy}.
It is defined as the invariant measure of a~walk made by a~random surfer
on the web graph. When reading a~given page, the surfer either
selects a~link from the current page (with a~uniform probability),
and moves to the page pointed by that link, or interrupts his current search,
and then moves to an~arbitrary page, which is selected
according to given ``zapping'' probabilities.
The rank of a~page is defined as its frequency of visit
by the random surfer. It is interpreted as the ``popularity'' of the page.
The PageRank optimization problem has been studied in several works:
\cite{AvrLit-OptStrat, Viennot-2006, NinKer-PRopt, IshiiTempo-FragileDataPRcomp, Blondel-PRopt, Fercoq-PRopt}.
The last two papers showed that PageRank optimization problems have a Markov decision process structure and both papers provided efficient algorithm that converge to a global optimum.
Cs\'aji, Jungers and Blondel in~\cite{Blondel-PRopt} showed that optimizing the PageRank score of a single web page
is a polynomial problem. Fercoq, Akian, Bouhtou and Gaubert in~\cite{Fercoq-PRopt} 
gave an alternative Markov decision process model and an efficient algorithm for the PageRank optimization problem 
with linear utility functions and more general design constraints, showing in particular that
any linear function of the PageRank vector can be optimized in polynomial time.

In this paper, we consider the more general situation in which the
ranking is determined
by the Perron eigenvector of a nonnegative, but not necessarily
stochastic, matrix. 
The Perron-Frobenius theorem (see~\cite{BermanPlemmons-nonnegMat} for instance) states that 
any nonnegative matrix $A$ has a nonnegative principal eigenvalue called the Perron root and nonnegative principal eigenvectors.
If, in addition, $A$ is irreducible, then the Perron root is simple and the (unique up to a multiplicative constant)
nonnegative eigenvector, called the Perron vector, has only positive entries.
This property makes it a good candidate to sort web pages. 
The ranking algorithms considered differ in the way of constructing from the web graph 
a nonnegative irreducible
matrix from which we determine the Perron vector. 
Then, the bigger is the Perron vector's coordinate corresponding
to a web page, the higher this web page is in the ranking.
In~\cite{Keener-FootbalTeams}, such a ranking
is proposed for football teams. The paper~\cite{Saaty-PerronRank} uses the Perron vector
to rank teachers from pairwise comparisons. 
See also~\cite{Vigna-spectralRanking} for a survey on the subject. When it comes to web page ranking,
the PageRank is the Perron eigenvector of the transition matrix described above
but the HITS and SALSA algorithms also rank pages according to a Perron vector.
% that we introduce next.

The HITS algorithm~\cite{Kleinberg-HITS} is not purely a link-based algorithm.
It is composed of two steps and the output depends on the query of the user.
Given a query, we first select a seed of pages that are relevant to the query according to
their text content. This seed is then extended with pages linking to them, pages to which they link
and all the hyperlinks between the pages selected.
We thus obtain a subgraph of the web graph focused on the query.
Then, the second step assigns each page two scores: a hub score $v$ and an authority score $u$ such that
good hubs should point to good authorities and good authorities should be pointed to by good hubs.
Introducing the adjacency matrix $A$ of the focused graph,
this can be written as $v=\rho Au$ and $u=\rho A^T v$ with $\rho \in \mathbb{R}_{+}$,
which means that the vector of hub scores is the Perron eigenvector of the matrices $A^T A$
and that the vector of authority scores is the Perron eigenvector of $AA^T$.
% The second step of the algorithm consists in finding the principal singular
% vectors of the adjacency matrix. The left singular vector is called the vector of hub scores
% or hub vector and the right singular vector is called the vector of authority scores or authority vector.
% Good hubs should point to good authorities and good authorities should be pointed to by good hubs.
% Note finally that the singular vectors of a matrix~$A$ are
% the eigenvectors of the matrices $A^T A$ and $A A^T$.
The construction of HITS' focused subgraph is a combination of text content relevancy with the query
and of hyperlink considerations. Maximizing the probability of appearance of a web page on this subgraph
is thus a composite problem out of the range of this paper. We shall however
study the optimization of HITS authority, for a given focused subgraph.
% In this paper, we shall concentrate on the second step of the algorithm,
% because maximizing the probability of appearance of a web page in
% the focused web graph is a problem of semantics, where variables are
% keywords rather than the web graph.

The SALSA algorithm~\cite{Lempel-Salsa} shares the same first step as HITS.
The second step consists in the computation of the invariant measure
of a stochastic matrix which consists in a normalization of the rows of $A^T A$.
In fact, with natural assumptions, this measure is proportionnal to the indegree of the web page. The authors
show that the interest of the ranking algorithm lies in the combination of the
two steps and not in one or the other alone. Thus from a hyperlink point of view,
optimizing the rank in SALSA simply consists in maximizing the number of hyperlinks
pointing to the target page. We shall not study SALSA optimization any further.

We also studied the optimization of Tomlin's HOTS scores~\cite{Tomlin-HOTS}.
In this case, the ranking is the vector of dual variables of an optimal flow problem. The flow
represents an optimal distribution of web surfers on the web graph in the sense of
entropy minimization. The dual variable, one by page, is interpreted as the ``temperature'' of
the page, the hotter a page the better.
Tomlin showed that this vector is solution of a nonlinear fix point equation: it may be seen
as a nonlinear eigenvector. Indeed, we show that most of the arguments available in the case
of Perron vector optimization can be adapted to HOTS optimization.
We think that this supports Tomlin's remark that
''malicious manipulation of the dual values of a large scale
nonlinear network optimization model [\ldots] would be an interesting topic``.
% which has
% not been studied, to our knowledge. Clearly, this would be an interesting
% topic for further research.''

\paragraph{Contribution}

In this paper, we study the problem of optimizing the Perron eigenvector
of a controlled matrix and apply it to PageRank, HITS and HOTS optimization.
Our first main result is the development of a scalable algorithm
for the local optimization of a scalar function of the Perron eigenvector over a set of nonnegative irreducible matrices.
Indeed, the global Perron vector optimization over a convex set of nonnegative matrices is
NP-hard, so we focus on the searching of local optima.
We give in Theorem~\ref{thm:fastDer} a power-type algorithm for the computation of the matrix of the partial
derivatives of the objective, based on the fact
that it is a rank 1 matrix. This theorem shows that computing the partial derivatives of
the objective has the same computational cost as computing the Perron vector by the power method,
which is the usual method when dealing with the large and sparse matrices built from the web graph.
Then we give an optimization algorithm that couples power and gradient iterations (Algorithms~\ref{alg:ArmijoApprox} and~\ref{alg:Master}).
Each step of the optimization algorithm involves a suitable number of power
iterations and a descent step. By considering this algorithm to be an approximate
projected gradient algorithm~\cite{Polak-consistent, Pironneau-ApproxGrad}, we prove that the algorithm converges
to a stationary point (Theorem~\ref{thm:approxGrad}). Compared with
the case when the number of power iterations is not adapted dynamically, we 
got a speedup between 3 and 20 in our numerical experiments (Section~\ref{sec:num})
together with a more precise convergence result.
% The proof uses %We also tune the algorithm thanks
% an estimation of the distance between the actual iterate of the power iterations
% and the eigenvector.

Our second main result is the application of Perron vector optimization to the optimization
of scalar functions of HITS authority or HOTS scores.
We derive optimization algorithms and, 
thanks to the low rank of the matrix of partial derivatives, we show that the
optimal linkage strategies of both problems satisfy a threshold property (Propositions~\ref{prop:shapeHits} and~\ref{prop:HotsOptStrats}).
This property was already proved for PageRank optimization in~\cite{NinKer-PRopt, Fercoq-PRopt}.
As in~\cite{IshiiTempo-FragileDataPRcomp, Blondel-PRopt, Fercoq-PRopt}
we partition the set of potential links $(i,j)$ into three subsets, consisting respectively of the set of \NEW{obligatory links}, the set of
\NEW{prohibited links} and the set of \NEW{facultative links}.
When choosing a subset of the facultative links, we get
a graph from which we get any of the three ranking vectors.
We are then looking for the subset of facultative links
that maximizes a given utility function.
We also study the associated relaxed problems, where we accept
weighted adjacency matrices. This assumes that the webmaster can influence
the importance of the hyperlinks 
of the pages she controls, for instance by choosing the size of the font, 
the color or the position of the link within the page.
In fact, we shall solve the relaxed problems
and then give conditions or heuristics to get an admissible
strategy for the discrete problems.

\paragraph{Related works}

As explained in the first part of the introduction, 
this paper extends the study of PageRank optimization developped
in~\cite{AvrLit-OptStrat, Viennot-2006, NinKer-PRopt, IshiiTempo-FragileDataPRcomp, Blondel-PRopt, Fercoq-PRopt}
to HITS authority~\cite{Kleinberg-HITS} and HOTS~\cite{Tomlin-HOTS} optimization.

We based our study of Perron eigenvector optimization on two other domains: eigenvalue optimization and
eigenvector sensitivity.
% The article deals with nonsmooth convex problems whereas we deal with
% smooth nonconvex problems: the techniques are thus totally different, even
% if the goals are similar.
There is a vast literature on eigenvalue and eigenvector sensitivity with many domains of application 
(see the survey~\cite{Hafta-sensitivity} for instance).
These works cope with perturbations of a given system. They consider
general square matrices and any eigenvalue or eigenvector. 
They give the directional derivatives of the eigenvalue and eigenvector of a matrix
with respect to a given perturbation of this matrix~\cite{Nelson-eigenvectorDer, Meyer-eigenvectorDer}. 
Perron eigenvalue and eigenvector sensitivity was
developped in~\cite{DeuNeu-PerronRootDer, DeuNeu-PerronVectorDer}.

This led to the development of eigenvalue optimization.
In~\cite{cullum-eigenValueOpt, Overton-largeScale, shapiro-eigenValueOpt} 
the authors show that the minimization of a convex 
function of the eigenvalues of symmetric matrices subject to linear constraints is a convex problem
and can be solved with semi-definite programming.
Eigenvalue optimization of nonsymmetric matrices is a more difficult problem.
In general, the eigenvalue is a nonconvex nonlipschitz function of the entries of the matrix.
The last section of~\cite{LewisOverton-eigenvalueOpt} proposes
a method to reduce the nonsymmetric case to the symmetric case 
by adding (many) additional variables.
Another approach is developped in~\cite{Overton-nonsymEigenValueOpt}:
the author derives descent directions and optimality conditions from 
the perturbation theory and uses so-called dual matrices.

In the context of population dynamics, the problem of the 
maximization of the growth rate of a population can
be modeled by the maximization of the Perron value 
of a given family of matrices.
This technique is used in~\cite{Logofet-calibration}
to identify the parameters of a population dynamic model, in~\cite{BiCFerGLOu-short} for
chemotherapy optimization purposes.
An approach based on branching Markov decision processes is presented in~\cite{Rothblum1982}. 
Perron value optimization also appears in other contexts
like in the minimization of the interferences in a network~\cite{Boche-perronvalueOpt}.

Apart from the stochastic case which can be solved by Markov decision techniques, like for PageRank, 
the search for a matrix with optimal eigenvectors does not seem to
have been much considered in the literature.
Indeed, the problem is not well defined since when an eigenvalue is not simple,
the associated eigenspace may have dimension greater than 1.

\paragraph{Organization}

The paper is organized as follows. In Section~\ref{sec:NPhard}, we introduce Perron eigenvector and eigenvalue optimization problems
and show that these problems 
 are NP-hard problems on convex sets of matrices. We also
point out some problems solvable in polynomial time.
In Section~\ref{sec:fastDer}, we give in Theorem~\ref{thm:fastDer} a~method for the efficient computation
of the derivative of objective function.
Then in Section~\ref{sec:approxGradient}, we give the coupled power and gradient iterations
and its convergence proof.
In Sections~\ref{sec:HITS} and \ref{sec:HOTS}, we show how HITS authority optimization
problems and HOTS optimization problems reduce to Perron vector optimization.
%In Section~\ref{sec:convEst}, we give effective 
%result verification proposition useful to the coupled power and gradient iterations.
Finally, we report numerical results on a fragment of the real web graph in Section~\ref{sec:num}.

\section{Perron vector and Perron value optimization problems} \label{sec:NPhard}

Let $M \in \mathbb{R}^{n \times n}$ be a (elementwise) nonnegative matrix. We say that $M$ is irreducible if
it is not similar to a block upper triangular matrix with two blocks via a permutation.
Equivalently, define the directed graph with $n$ nodes and an edge between node $i$ and $j$ if
and only if $M_{i,j}>0$: $M$ is irreducible if and only if this graph is strongly connected.

We denote by $\rho(M)$ the principal eigenvalue of the 
irreducible nonnegative matrix $M$, called the Perron root.
By Perron-Frobenius theorem (see~\cite{BermanPlemmons-nonnegMat} for instance), we know that
$\rho(M)>0$ and that this eigenvalue is simple.
Given a normalization $N$, we denote by $u(M)$ 
the corresponding normalized eigenvector, called the Perron vector.
The normalization is necessary since the Perron vector is only defined
up to positive multiplicative constant. The normalization function $N$
should be homogeneous and we require $u(M)$ to verify $N(u(M))=1$.
The Perron-Frobenius theorem asserts that $u(M)>0$ elementwise. 
The Perron eigenvalue optimization problem on the set $\mathcal{M}$ can be written as:
\begin{equation} \label{eq:perronEigOpt}
 \min_{M \in \mathcal{M}} f(\rho(M))
\end{equation}
The Perron vector optimization problem can be written as:
\begin{equation} \label{eq:perronVecOpt}
 \min_{M \in \mathcal{M}} f(u(M))
\end{equation}
We assume that $f$
is a real valued continuously differentiable function; $\mathcal{M}$ is a set of irreducible nonnegative matrices
such that $\mathcal{M}=h(\C)$ with $h$ continuously differentiable and $\C$ a closed convex set.
These hypotheses allow us to use algorithms such as projected gradient for the 
searching of stationary points.

We next observe that the minimization of the Perron root and the optimization
of a scalar function of the Perron vector are NP-hard problems and
that only exceptional special cases appear to be polynomial time solvable
by current methods. Consequently, we shall focus
on the local resolution of these problems, with an emphasis on large sparse problems
like the ones encountered for web applications.
We consider the two following problems:

PERRONROOT\_MIN: given a rational linear function $A: \mathbb{R}^{n \times n} \to \mathbb{R}^m$ and a vector $b$ in $\mathbb{Q}^m$,
find a matrix $M$ that minimizes $\rho(M)$
on the polyhedral set $\{ M \in \mathbb{R}^{n \times n} \; | \; A(M) \leq b \;,\; M \geq 0 \}$.

PERRONVECTOR\_OPT: given a rational linear function $A: \mathbb{R}^{n \times n} \to \mathbb{R}^m$, a vector $b$ in $\mathbb{Q}^m$, a rational function $f: \mathbb{R}^n \to \mathbb{R}$ and a rational normalization function $N$,
find a matrix $M$ that minimizes $f(u(M))$
on the polyhedral 
set $\{ M \in \mathbb{R}^{n \times n} \; | \; A(M) \leq b \;,\; M \geq 0 \}$,
where $u(M)$ verifies $N(u(M))=1$.

% \begin{align*}
%  \min \rho(M) & \mathrm{(P2)} \;& \min f(u(M)) &  \\
%  A(M) \leq b \;,\; M \geq 0 & & A(M)\leq b \;,\; M \geq 0 & 
% \end{align*}
% with inputs a rational linear function $A: \mathbb{R}^{n \times n} \to \mathbb{R}^m$ and a vector $b$ in $\mathbb{Q}^m$
% find a matrix $M \in \mathbb{R}^{n \times n}$ solution of:
% \begin{align*} 
% \mathrm{(P1)} \;& \min \rho(M) & \mathrm{(P2)} \;& \min f(u(M)) &  \\
% & A(M) \leq b \;,\; M \geq 0 & & A(M)\leq b \;,\; M \geq 0 & 
% \end{align*}

In general, determining whether all matrices in an 
interval family are stable is a NP-hard problem~\cite{BlondelTsitsiklis}.
The corresponding problem for nonnegative matrices is
to determine whether the maximal Perron root is smaller than a given number.
Indeed, as the Perron root is a monotone function of
the entries of the matrix (see Proposition~\ref{prop:eigenvalueDer} below), this problem is trivial on interval matrix families.
However, we shall prove NP-hardness of PERRONROOT\_MIN and PERRONVECTOR\_OPT by reduction
of linear multiplicative programming problem to each of these problems.
The linear multiplicative problem is:

LMP: given a $n \times m$ rational matrix $A$ and a vector $b$ in $\mathbb{Q}^m$,
find a vector $x \in \mathbb{R}^n$ that minimizes $x_1 x_2$
on the polyhedral set $\{ x \in \mathbb{R}^n \; | \; A x \leq b \;,\; x_1, x_2 \geq 0 \}$.

% solution of:
% \begin{align*}
% \mathrm{(P3)} \;& \min x_1 x_2 \\
% & Ax\leq b \;,\; x \geq 0 
% \end{align*}

A theorem of Matsui~\cite{Matsui-NPhardMultOpt} states that LMP is NP-hard.
We shall need a slightly stronger result about a weak version of LMP:

Weak-LMP: given $\epsilon>0$, a $n \times m$ rational matrix $A$ and a vector $b$ in $\mathbb{Q}^m$,
 find a vector $x \in Q$ such that $x_1 x_2 \leq y_1 y_2 +\epsilon$ for all $y \in Q$,
where $Q =\{ x \in \mathbb{R}^n \; | \; A x \leq b \;,\; x_1, x_2 \geq 0 \}$.

\begin{lemma}
Weak-LMP is a NP-hard problem.
\end{lemma}
\begin{proof}
A small modification of the proof of Matsui~\cite{Matsui-NPhardMultOpt}
gives the result.
If we replace $g(x_0, y_0) \leq 0$ by $g(x_0, y_0) \leq -2$ in Corollary~2.3 
we remark that the rest of the proof still holds 
since $n^4 p^{4n}+p^2 -4 p^{4n+1} \leq -2$ for all $n \geq 1$ and $p=n^{n^4}$.
Then, with the notations of~\cite{Matsui-NPhardMultOpt},
we have proved that the optimal value of $P1(M)$ is less than or
equal to $4p^{8n}-2$ if and only if $Mx=1$ has a $0-1$ valued solution
and it is greater than or equal to $4p^{8n}$ if and only if $Mx=1$ does not have a $0-1$ valued solution.
We just need to choose $\epsilon < 2$ in problem $P1(M)$ to finish the proof of the lemma.
\end{proof}

\begin{prop}
PERRONROOT\_MIN and PERRONVECTOR\_OPT are NP-hard problems.
\end{prop}
\begin{proof}
We define the matrix $M$ by the matrix with lower diagonal equal to $x$ and $1$ on the top left corner:
\[
 M=\begin{pmatrix}
    0 & 0 & \ldots & 0 & 1 \\
x_1 & 0 & 0  & & 0\\
0 & x_2 & 0 &  & \vdots \\
\vdots & & \ddots &\ddots& \vdots\\
0 & \ldots & 0 & x_n & 0
   \end{pmatrix} \enspace .
\]
We set the admissible set $X$ for the vector $x$ as
$X=\{ x \in \mathbb{R}^m | A x \geq b , x \geq 0\}$ with
a rational $p \times m$ matrix $A$ and a rational vector $b$ of length $p$.

For the eigenvector problem, we set the normalization
$u_1=1$ and we take $f(u)=u_{n}$. We have
$u_n=\rho$, so the complexity is the same for
eigenvalue and eigenvector optimization is this context.

Now minimizing $\rho(M)$ is equivalent to minimizing $x_1 x_2 \ldots x_n$
because the n-th root is an nondecreasing function on $\mathbb{R}_{+}$.
We thus just need to reduce in polynomial time the $\epsilon$-solvability
of weak-LMP to the minimization of $x_1 x_2 \ldots x_n$ on $X$.
% $X$ in polynomial time, then we could minimize $x_1 x_2$
% The former is exactly 
% so we have proved the NP-hardness of Perron eigenvalue and
% eigenvector optimization.

As $x \mapsto \log(x_1 x_2 \ldots x_n)$ is a concave function,
either the problem is unbounded or there exists an 
optimal solution that is an extreme point of the polyhedral admissible set (Theorem~3.4.7 in~\cite{Bazaraa-Nonlinopt}).
Hence, using Lemma~6.2.5 in~\cite{Lovasz-geomAlgo} we see that we can define 
a bounded problem equivalent to the unbounded problem and with coefficients that have a polynomial encoding length:
in the following, we assume that the admissible set is bounded.

Given a rational number $\epsilon>0$, a rational $p \times m$ matrix $A'$ and a rational vector $b'$ of length $p$,
let $X':=\{ x \in \mathbb{R}^m | A' x \geq b' , x_1 \geq 0, x_2 \geq 0\}$ be a bounded polyhedron.
Denoting $v_2:= \min_{x \in X'} x_1 x_2$, we are looking for
 $x' \in X'$ such that $\abs{y_1 y_2 - v_2} \leq \epsilon$.

Compute $\underline{C}:=\min_{i \in [n]} \min_{x \in X'} x_i$ and $\bar{C}:=\max_{i \in [n]} \max_{x \in X'} x_i$
(linear programs). 
We first set $m_0=\bar{C}-\underline{C}+1$ so that $m_0>0$ and $m_0>-\underline{C}$ 
and $t^{m_0} \in \mathbb{R}^n$ defined by $t^{m_0}_i=0$ if $i\in \{1,2\}$ and $t^{m_0}_i=m$ if $i \geq 3$.
Let 
\[
X_0:=\{ x \in \mathbb{R}^m | A' x \geq b'-A' t^{m_0} , x \geq 0\} \enspace .
\]
We have $v_2= \min_{x \in X'} x_1 x_2=\min_{x \in X_0} x_1 x_2$. 
Let $v_0:=\min_{x \in X} x_1 x_2 \ldots x_n$. For all $x \in X_0$, we have
\[
 x_1 x_2 (m_0+\underline{C})^{n-2} \leq x_1 x_2 \ldots x_n \leq  x_1 x_2 (m_0+\bar{C})^{n-2} \enspace ,
\]
so that $v_2 (m_0+\underline{C})^{n-2} \leq v_0 \leq v_2 (m_0+\bar{C})^{n-2}$.

We now set 
\[
m:=\max \{ -\underline{C}+2^{n-3} \frac{\bar{C}}{\epsilon} \frac{v_0}{(m_0+\underline{C})^{n-2}} , \bar{C} - 2 \underline{C} , 1 \}
\]
 and we define $t^m$ and $X$ in the same way as $t^{m_0}$ and $X_0$.
Remark that $m$ an encoding length polynomial in the length of the entries.
Let $v_n:=\min_{x \in X} x_1 x_2 \ldots x_n$. For all $x \in X$, we have
$v_2 (m+\underline{C})^{n-2} \leq v_n \leq v_2 (m+\bar{C})^{n-2}$ and $x' = x -t^m$ is a point of $X'$ with $x'_1 x'_2=x_1 x_2$.

As $v_2 (m_0+\underline{C})^{n-2} \leq v_0$, $m \geq -\underline{C}+2^{n-3} \frac{\bar{C}}{\epsilon} v_2$.
As $m \geq \bar{C} - 2 \underline{C}$, $\frac{m+\underline{C}}{m+\bar{C}} \geq \frac{1}{2}$,
 so that $\frac{(m+\underline{C})^{n-2}}{(m+\bar{C})^{n-3}} \geq \frac{1}{2^{n-3}}(m+\underline{C})$ and 
$\frac{(m+\underline{C})^{n-2}}{(m+\bar{C})^{n-3}} \geq \frac{\bar{C}}{\epsilon} v_2$.

Denote $M:=(m+\underline{C})^{n-2}$ and $\Delta:= (m+\bar{C})^{n-2} - (m+\underline{C})^{n-2}$: $M v_2 \leq v_n \leq (M+\Delta) v_2$
We have $\Delta:= \sum_{k=0}^{n-3} \binom{n-2}{k} m^k \bar{C}^{n-2-k} - \sum_{k=0}^{n-3} \binom{n-2}{k} m^k \underline{C}^{n-2-k}
 \leq \bar{C} (m+\bar{C})^{n-3}$. Hence, $M \geq  \frac{\Delta}{\epsilon} v_2$.
As $\Delta \geq 0$, $M+\Delta \geq  \frac{\Delta}{\epsilon} v_2$. We obtain
\[
 \epsilon \geq \frac{\Delta M v_2}{M(M+\Delta)}= (\frac{1}{M}-\frac{1}{M+\Delta})M v_2 \geq v_2 - \frac{v_n}{M+\Delta} \enspace .
\]
Finally
\[
 \frac{v_n}{ (m_0+\bar{C})^{n-2}}\leq v_2 \leq \frac{v_n}{(m_0+\bar{C})^{n-2}} +\epsilon \enspace .
\]
% Theorem~6.3.2 in~\cite{Lovasz-geomAlgo} weak implies strong for PL
which proves that weak-LMP reduces to the minimization of $x_1 x_2 \ldots x_n$ on $X$.
\end{proof}

The general Perron eigenvalue optimization problem is NP-hard
but we however point out some cases for which it is tractable.
The following proposition is well known:
\begin{prop}[\cite{Kingman-logconv}]
 The eigenvalue $\rho(M)$ is a log-convex function of the log of the entries of the nonnegative matrix $M$.
\end{prop}
This means that $\log \circ \rho \circ \exp$ is a convex function, where $\exp$ is the 
componentwise exponential, namely if $0 \leq \alpha \leq 1$ and $A$ and $B$ are two nonnegative $n \times n$ matrices
then for $C_{i,j}=A_{i,j}^\alpha B_{i,j}^{1-\alpha}$, $\rho(C) \leq \rho(A)^\alpha \rho(B)^{1-\alpha}$.
\begin{cor}
The optimization problem
 \begin{equation*} 
 \min_{M \in \exp(\C)} \rho(M)
\end{equation*}
with $\C$ convex is equivalent to the convex problem
 \begin{equation*}
 \min_{L \in C} \log \circ \rho \circ \exp(L)
\end{equation*}
\end{cor}
The difference between this proposition and the previous one is that here $\mathcal{M}=h(C)=\exp(C)$ whereas
previously we had $h$ affine. This makes a big difference
since an $\epsilon$-solution of a convex program can be found in polynomial time~\cite{Nem-modernConvOpt}.

\begin{remark}
The largest singular value (which is a norm) is a convex function of the entries of the matrix.
For a symmetric matrix, the singular values are the absolute values of the eigenvalues.
Thus minimizing the largest eigenvalue on a convex set of nonnegative symmetric matrices is
a convex problem.
\end{remark}

In order to solve the signal to interference ratio balancing problem,
Boche and Schuber~\cite{Boche-perronvalueOpt}
give an algorithm for the global minimization of the Perron root when
the rows of the controlled matrix are independently controlled, ie when the
admissible set is of the form $\mathcal{Z}_1 \times \ldots \times \mathcal{Z}_n$
and $z_k \in \mathcal{Z}_k$ is the $k^{th}$ row of the matrix.
\begin{prop}[\cite{Boche-perronvalueOpt}] \label{prop:decoupled}
Let $\mathcal{Z}=\mathcal{Z}_1 \times \ldots \times \mathcal{Z}_n$
and $\Gamma$ be a fixed positive diagonal matrix.
If the $k^{th}$ row of the matrix $V(z)$ only depends on $z_k \in \mathcal{Z}_k$,
$V(z)$ is irreducible for all $z \in \mathcal{Z}$
and if $V(z)$ is continuous on $\mathcal{Z}$ ($\mathcal{Z}$ can be discrete),
then there exists a monotone algorithm that minimizes
$\rho(\Gamma V(z))$ over $\mathcal{Z}$, in the sense that
$\rho(\Gamma V(z_{n+1})) \leq \rho(\Gamma V(z_{n}))$ for all $n \geq 0$
and $\lim_n \rho(\Gamma V(z_{n})) = \min_{z \in \mathcal{Z}} \rho(\Gamma V(z))$.
\end{prop}
Let $X$ be the admissible set of a Perron eigenvalue optimization problem.
Denote $\pi_i$ the projection on the $i^{th}$ coordinate.
Then the minimization of the Perron value over $\tilde{X}:=\pi_1(X) \times \ldots \times \pi_n(X)$
is a relaxation of the original problem which is solvable with
the algorithm of Proposition~\ref{prop:decoupled} as soon as all matrices in $\tilde{X}$ are irreducible.
We thus get a lower bound for the optimization of the Perron eigenvalue problem
in a general setting.
This is complementary with the local optimization approach developped in this article, which would yield
an upper bound on the optimal value of this problem.

\begin{remark} \label{rem:PRopt}
As developped in~\cite{Fercoq-PRopt}, general PageRank optimization problems can be formulated as follows.
Let $M$ be the transition matrix of PageRank and $\rho$ the associated
occupation measure. When $M$ is irreducible, they are linked by
the relation $M_{ij}=\frac{\rho_{ij}}{\sum_k \rho_{ik}}=h_{ij}(\rho)$, which yields our function $h$.
We also have $u(h(\rho))_i=\sum_k \rho_{ik}$ and $\rho_{ij}=u(M)_i M_{ij}$.

If the set $\C$, which defines the design constraints of the webmaster,
is a convex set of occupation measures, if $h$ is as defined above and 
if $f$ is a convex function, then 
\begin{equation*} \label{eq:pagerankOpt}
 \min_{\rho \in \C} f(u(h(\rho)))
\end{equation*}
is a convex problem.
Thus $\epsilon$-solutions of PageRank optimization problems can be found in polynomial-time.
Details of this development and exact resolution for a linear $f$ can be found in~\cite{Fercoq-PRopt}.
\end{remark}

\section{A power-type algorithm for the evaluation of the derivative of a function of the principal eigenvector} \label{sec:fastDer}

We now turn to the main topic of this paper. We give in Theorem~\ref{thm:fastDer} a power-type algorithm for the evaluation
of the partial derivatives of the principal eigenvector of a matrix with a simple principal eigenvalue.

We consider a matrix $M$ with a simple eigenvalue $\lambda$ and associated
left and right eigenvectors $u$ and $v$. We shall normalize $v$ by the assumption $\sum_{i \in [n]} v_i u_i=1$.
The derivatives of the eigenvalue of a matrix are well
known and easy to compute:
\begin{prop}[\cite{Kato} Section~II.2.2] \label{prop:eigenvalueDer}
Denoting $v$ and $u$ the left and right eigenvectors of a matrix $M$ associated to a simple eigenvalue $\lambda$,
normalized such that $\sum_{i \in [n]} v_i u_i=1$, 
the derivative of $\lambda$ can be written as:
\begin{equation*}
 \frac{\partial \lambda}{\partial M_{ij}} = v_i u_j
\end{equation*}
\end{prop}
% \begin{remark}
% This property and the positivity of Perron eigenvectors show that Perron eigenvalue maximization (respectively minimization) on rectangles simply consists
% in choosing the maximal (respectively minimal) feasible matrix.
% \end{remark}

In this section, we give a scalable algorithm
to compute the partial derivatives of the function $f \circ u$, %$M \mapsto \sum_k c_k u_k(M)$,
that to an irreducible nonnegative matrix $M$ associates the utility of its Perron vector.
In other words we compute $g_{ij}=\sum_k \frac{\partial f}{\partial u_{k}} \frac{\partial u_k}{\partial M_{ij}}$.
%Thanks to the chain rule, we can then get the gradient of $f \circ u \circ h$.
%We want scalable algorithms because we are
%likely to use them for large problems. The main result of this section is
This algorithm is a sparse iterative scheme and it is the core of the optimization algorithms %that computes $\nabla (f \circ u)$.
that we will then use for the large problems encountered in the optimization of web ranking.

We first recall some results on the derivatives of eigenprojectors (see~\cite{Kato} for more background).
Throughout the end of the paper, we shall consider column vectors
and row vectors will be written as the transpose of a column vector.
Let $P$ be the eigenprojector of $M$ for the eigenvalue $\lambda$.
One can easily check that as $\lambda$ is a simple eigenvalue, the spectral projector is $P = u v^T$ as soon as $v^T u=1$.
We have the relation
\[
 \frac{\partial P}{\partial M_{ij}} = -S E_{ij} P - P E_{ij} S
\]
 where $E_{ij}$ is the $n \times n$ matrix with all entries zero except the $ij^{th}$
and $S=(M-\lambda I)^{\#}$ is the Drazin pseudo-inverse of $M-\lambda I$. This matrix $S$ also satisfies the equalities
\begin{equation} \label{eqn:S}
 S(M-\lambda I)=(M-\lambda I)S=I-P \quad \text{and} \quad S P = P S = 0
\end{equation}

When it comes to eigenvectors, we have to set a normalization for each of them.
Let $N$ be the normalization function for the right eigenvector.
We assume that it is differentiable at $u$ and that $N(\alpha u)=\alpha N(u)$ for all nonnegative scalars $\alpha$,
which implies $\frac{\partial N}{\partial u}(u) \cdot u = N(u)$.
We normalize $v$ by the natural normalization $\sum_i u_i v_i =1$.

\begin{prop}[\cite{Meyer-eigenvectorDer}] \label{prop:eigenvectorDer}
Let $M$ be a matrix with a simple eigenvalue $\lambda$ and associated eigenvector $u$
normalized by $N(u)=1$. We denote $S=(M-\lambda I)^{\#}$. 
Then the partial derivatives of the eigenvector are given by:
\begin{equation*}
\frac{\partial u}{\partial M_{ij}}(M)= -S e_i u_j + (\nabla N(u)^T S e_i u_j)u
\end{equation*}
where $e_i$ is the vector with $i^{th}$ entry equal to 1.
\end{prop}
To simplify notations, we denote $\nabla f$ for $\nabla f(u)$ and $\nabla N$ for $\nabla N(u)$.
\begin{cor} \label{prop:gcool} 
Let $M$ be a matrix with a simple eigenvalue $\lambda$ and associated eigenvector $u$
normalized by $N(u)=1$.
The partial derivatives of the function $M \mapsto (f \circ u) (M)$
at $M$ are $g_{ij}=w_i u_j$, where the auxiliary vector $w$ is given by:
\begin{equation*}
 w^T=(-\nabla f^T + (\nabla f \cdot u) \nabla N^T) S = (-\nabla f^T + (\nabla f \cdot u) p^T) (M-\lambda I)^{\#}
\end{equation*}
\end{cor}
\begin{proof}
By Proposition \ref{prop:eigenvectorDer}, we deduce that
\begin{align*}
 g_{ij}= \sum_k \frac{\partial f}{\partial u_k}(u(M))\frac{\partial u_k}{\partial M_{ij}}(M)
= -\sum_k \frac{\partial f}{\partial u_k} S_{ki} u_j + \sum_k\frac{\partial f}{\partial u_k} u_k \sum_l \frac{\partial N}{\partial u_l} S_{li} u_j
\end{align*}
which is the developed form the result.
\end{proof}
This simple corollary already improves the computation speed. Using Proposition \ref{prop:eigenvectorDer}
directly, one needs to compute $\frac{\partial P}{\partial M_{ij}}$ in every direction, which means 
the computation of a Drazin inverse and
 $2 n^2$ matrix-vector products involving $M$ for the cmputation of $\left (\frac{\partial u}{\partial M_{ij}} \right)_{i,j \in [n]}$.
With Corollary \ref{prop:gcool} we only need one matrix-vector
product for the same result. The last difficulty is the computation of
the Drazin inverse $S$. In fact, we do not need to compute the 
whole matrix but only to compute $S^Tx$ for a given $x$. The
next two propositions show how one can do it.

\begin{prop}
Let $M$ be a matrix with a simple eigenvalue $\lambda$ and associated eigenvector $u$
normalized by $N(u)=1$.
The auxiliary vector $w$ of Corollary~\ref{prop:gcool} is solution of the following
invertible system
\begin{equation} \label{eqn:wSys}
 [w^T, w_{n+1}] \begin{bmatrix}M-\lambda I & -u \\ \nabla N^T & 0  \end{bmatrix} = [-\nabla f^T , 0 ] \enspace.
\end{equation}
 where $w_{n+1} \in \mathbb{R}$.
\end{prop}
\begin{proof}
A nullspace argument similar to this of \cite{Nelson-eigenvectorDer} 
shows that $\begin{bmatrix}M-\lambda I & -u \\ \nabla N^T & 0  \end{bmatrix}$
is invertible as soon as $\lambda$ is simple and $\nabla N^T u=1$.
Then the solution $w$ of the system~\eqref{eqn:wSys} verifies the equations $w^T (M-\lambda I) + w_{n+1} \nabla N ^T = -\nabla f^T$ and $w^T u=0$.
Multiplying the first equality by $u$ yields $w_{n+1} \nabla N ^Tu=-\nabla f^T u$
and multiplying it by $S$ yields $w^T (I - u v^T) = - w_{n+1} \nabla N ^T S - \nabla f^T S$.
Putting all together, we get $w^T=(- \nabla f^T +(\nabla f^T u) \nabla N ^T )S$.
\end{proof}

The next proposition provides an iterative scheme to compute the evaluation
of the auxiliary vector~$w$ when we consider the principal eigenvalue.

\begin{defn}
 We say that a sequence $(x_k)_{k \geq 0}$ converges to a point $x$ with a linear convergence rate
$\alpha$ if $\limsup_{k\to \infty} \norm{x_k - x}^{1/k} \leq \alpha$.
\end{defn}

\begin{prop} \label{prop:wFast}
Let $M$ be a matrix with only one eigenvalue of maximal modulus $\rho=\abs{\lambda_1}>\abs{\lambda_2}$.
With the same notations as in Corollary~\ref{prop:gcool}, we denote $\tilde{M}=\frac{1}{\rho}M$ and $z^T=\frac{1}{\rho}(-\nabla f^T + (\nabla f \cdot u) \nabla N^T)$, 
and we fix a real row vector~$w_0$.
Then the fix point scheme defined by
\begin{equation*}
\forall k \in \mathbb{N},\quad w^T_{k+1}= (-z^T + w^T_k \tilde{M})(I-P)
\end{equation*}
with $P=u v^T$, converges to $w^T=(-\nabla f^T + (\nabla f \cdot u) \nabla N^T) (M-\rho I)^{\#}$
with a linear rate of convergence $\frac{\abs{\lambda_2}}{\rho}$.
\end{prop}
\begin{proof}
 We have $w^T_{k}= \sum_{l=0}^{k-1} -z^T(\tilde{M} (I-P))^l + w^T_0 (\tilde{M} (I-P))^k$.
 By assumption, all the eigenvalues of $\tilde{M}$ different from $1$ have a modulus smaller than~$1$. Thus, 
using the fact that $P$ is the eigenprojector associated to $1$, we get $\rho(\tilde{M} (I-P))=\frac{\abs{\lambda_2}}{\rho}<1$.
By~\cite{Ostrowski}, $(\norm{\tilde{M} (I-P)}^k)^{1/k} \to \rho(\tilde{M} (I-P))$, so
the algorithm converges to a limit $w$ and for all $\epsilon>0$, $\norm{w_k-w} = \bigO( (\frac{\abs{\lambda_2} + \epsilon}{\rho})^k )$.
This implies a~linear convergence rate equal to $\frac{\abs{\lambda_2}}{\rho}$.
The limit $w$ satisfies $w^T= (-z^T + w^T \tilde{M})(I-P)$, so $w^TP=0$ and as $\tilde{M}P=P$, $w^T\tilde{M}-w^T=z^T(I-P)$.
We thus get the equality $w^T(\tilde{M}-I)=z^T$.
Multiplying both sides by $(\tilde{M}-I)^{\#}$, we get:
\begin{equation*}
  w^T(\tilde{M}-I)(\tilde{M}-I)^{\#}=w^T-w^T P=w^T=z^T(\tilde{M}-I)^{\#}
\end{equation*}
The last equalities and the relation $(\beta^{-1} M)^{\#}=\beta M^{\#}$ show by Proposition~\ref{prop:gcool} that $g=w u^T$ is 
the matrix of partial derivatives of the Perron vector multiplied by~$\nabla f$.
\end{proof}
This iterative scheme uses only matrix-vector products and thus may be very efficient
for a sparse matrix. In fact, it has the same linear convergence rate
as the power method for the computation of the Perron eigenvalue and eigenvector.
This means that the computation of the derivative of the eigenvector has
a computational cost of the same order as the computation of the eigenvector itself.
We next show that the eigenvector and its derivative can be
computed in a single algorithm.

\begin{thm} \label{thm:fastDer}
 If $M$ is a matrix with only one simple eigenvalue of maximal modulus $\rho=\abs{\lambda_1}>\abs{\lambda_2}$,
then the derivative $g$ of the function $f \circ u$ at $M$, such that 
$g_{ij}=\sum_k \frac{\partial f}{\partial u_{k}} \frac{\partial u_k}{\partial M_{ij}}$
is the limit of the sequence $(\tilde{w}_l u_l^T)_{l \geq 0}$ given by the following iterative scheme:
\begin{align*}
 u_{l+1}&=\frac{M u_l}{N(M u_l)}\\
 v_{l+1}^T&=\frac{v_l^T M}{v_l^T M u_{l+1}}\\
 \tilde{w}_{l+1}^T&=\frac{1}{\rho_l}(\nabla f_l^T - (\nabla f_l \cdot u_l) \nabla N_l^T + \tilde{w}_l^T M) (I-u_{l+1} v_{l+1}^T)
\end{align*}
where $\rho_l=N(M u_l)$, $\nabla f_l=\nabla f(u_l)$ and $\nabla N_l=\nabla N(u_l)$.
Moreover, the sequences $(u_l)$, $(v_l)$ and $(\tilde{w}_l)$ converge linearly with rate $\frac{\abs{\lambda_2}}{\rho}$.
\end{thm}
Of course, the first and second sequences are the power method to the right and to the left.
The third sequence is a modification the scheme of Proposition~\ref{prop:wFast}
with currently known values only. We shall denote one iteration of the
scheme of the theorem as 
\[
(u_{k+1}, v_{k+1}, \tilde{w}_{k+1})=\mathrm{POWERDERIVATIVE_M}(u_k, v_k, \tilde{w}_k) \enspace . 
\]

\begin{proof}
The equalities giving $u_{l+1}$ and $v_{l+1}$ are simply the usual power method, so by \cite{ParlettPoole-power},
they convergence linearly with rate $\frac{\abs{\lambda_2}}{\rho}$
to $u$ and $v$, the right and left principal eigenvectors of $M$, such that $P=uv^T$ is the eigenprojector associated to $\rho(M)$.
Let 
\[
z_l^T:=\frac{1}{\rho_l}(-\nabla f_l^T + (\nabla f_l \cdot u_l) \nabla N_l^T) \enspace:
\]
$\lim z_l = z$ by continuity of $\nabla f$, $N$ and $\nabla N$ at $u$. We also have
\[
 \tilde{w}_l^T=(-z_l^T+\frac{1}{\rho_l} \tilde{w}_l^T M)(I-u_{l+1}v_{l+1}^T) \enspace .
\]
We first show that $\tilde{w}_l$ is bounded. As in the proof of Proposition~\ref{prop:wFast}, 
$\rho(\tilde{M}(I-P))=\frac{\abs{\lambda_2}}{\rho}<1$.
Thus, by Lemma~5.6.10 in~\cite{HornJohnson-MatrixAnalysis}, there exists a norm $\norm{\cdot}_M$ and $\alpha<1$ such that
$\tilde{M}(I-P)$ is $\alpha$-contractant. Let $S$ be the unit sphere: $\forall x \in S, \norm{\tilde{M}(I-P)x}_M \leq \alpha \norm{x}_M$.
By continuity of the norm, $\forall \epsilon >0, \exists L, \forall l\geq L, \forall x \in S,
\norm{\frac{1}{\rho_l}M(I-u_{l+1}v_{l+1}^T)x}_M \leq (\alpha+\epsilon) \norm{x}_M$.
As $\frac{1}{\rho_l}M(I-u_{l+1}v_{l+1}^T)$ is linear, we have the result on the whole $\mathbb{R}^n$ space.
Thus $\tilde{w}_l$ is bounded.

Let us denote $\tilde{M}=\frac{1}{\rho}M$ and
\[
 \tilde{z}_l^T:=z_l^T (I-u_{l+1}v_{l+1}^T)+\frac{1}{\rho_l}\tilde{w}_l^T M (u v^T-u_{l+1}v_{l+1}^T)+
\frac{\rho-\rho_l}{\rho \rho_l}\tilde{w}_l^T M (I-uv^T)\enspace,
\]
so that $\tilde{w}_{l+1}^T=-\tilde{z}_l^T+\tilde{w}_l^T \tilde{M}(I-uv^T)$.
% The equalities giving $u_{l+1}$ and $v_{l+1}$ are simply the usual power method, so by \cite{ParlettPoole-power},
% they imply a linear convergence of rate $\frac{\abs{\lambda_2}}{\rho}$
% for $u_{l}$ and $v_{l}$. Thus $\lim z_l = z$ by continuity of $\nabla f$, $N$ and $\nabla N$ at $u$.
% We first show that $\tilde{z}_l$ is bounded.
% If $\norm{\tilde{w}_l}$ is bounded, the result is trivial and we even have the convergence of the sequence $(\tilde{z}_l)$ to $z$.
% Otherwise, we have:
% \begin{align*}
%  \frac{\norm{\tilde{z}_l}}{\norm{\tilde{w}_l}} \leq \frac{\norm{z_l^T (I-u_{l+1}v_{l+1}^T)}}{\norm{\tilde{w}_l}}+\frac{1}{\rho_l}\norm{M (u v^T-u_{l+1}v_{l+1}^T)}+\frac{\abs{\rho-\rho_l}}{\rho \rho_l}.
% \end{align*}
%  The right part of the inequality is equivalent to $\frac{\norm{z}}{\norm{w_l}}$. Multiplying back by $\norm{w_l}$
% we get $\norm{\tilde{z}_l} \in O(1)$.
%
%The equality $\tilde{w}_{l+1}^T=-\tilde{z}_l^T+\tilde{w}_l^T \tilde{M}(I-uv^T)$ gives:
We have:
\begin{align*}
 \tilde{w}_l^T&= \tilde{w}_0^T (\tilde{M} (I-P))^l  - \sum_{k=0}^{l-1} \tilde{z}_{l-1-k}^T (\tilde{M} (I-P))^k \\
%    & = \left ( w_0^T (\tilde{M} (I-P))^{l-r} - \sum_{k=r}^{l-1} \tilde{z}_{l-1-k}^T (\tilde{M} (I-P))^{k-r} \right ) (\tilde{M} (I-P))^{r}\\
%    &     - \sum_{k=0}^{r-1} z^T (\tilde{M} (I-P))^{k} - \sum_{k=0}^{r-1} (\tilde{z}_{l-1-k}^T-z^T) (\tilde{M} (I-P))^{k}
& =\tilde{w}_0^T (\tilde{M} (I-P))^{l} - \sum_{k=0}^{l-1} z^T (\tilde{M} (I-P))^{k} - \sum_{k=0}^{l-1} (\tilde{z}_{l-1-k}^T-z^T) (\tilde{M} (I-P))^{k}
\end{align*}

 %$\tilde{z}_l$ is bounded, and a
%Choose $\epsilon>0$.
By Proposition~\ref{prop:wFast}, the sum of the first and second summand correspomd to $w_l$ and converge linearly to $w^T$ when $l$ tends to infinity with 
convergence rate $\frac{\abs{\lambda_2}}{\rho}$. Corollary~\ref{prop:gcool} asserts that $g=\lim w_l u_l^T$.
% 
% For the last one, as $\norm{(\tilde{M} (I-P))^{k}}$ is equivalent to $(\frac{\abs{\lambda_2}}{\rho})^k$,
% the sum $\sum_{k=0}^{l-1} \norm{(\tilde{M}(I-P))^k}$ is bounded. 
% %note that $\sum_{k=0}^{r-1} \norm{(\tilde{M}(I-P))^k}$ is bounded and 
% As $\norm{\tilde{z}_l-z}$ is also bounded ($\tilde{z}_l$ is bounded), we have proved that $(\tilde{w}_l)_{l\geq0}$
% is bounded and thus that $\lim \tilde{z}_l=z$.
% 
For the last one, we remark that for all $\epsilon>0$,
$\norm{(\tilde{M} (I-P))^{k}} = \bigO(\frac{\abs{\lambda_2 +\epsilon }}{\rho})^k$.
In order to get the convergence rate of the sequence, we need to estimate $\norm{\tilde{z}_l-z} = \norm{\tilde{z}_l -z_l + z_l-z}$.
\[
 \norm{\tilde{z}_l-z} \leq \norm{(z_l^T u_{l+1}) v_{l+1}^T} + \norm{\frac{1}{\rho_l}\tilde{w}_l^T M (u v^T-u_{l+1}v_{l+1}^T)} + \norm{\frac{\rho-\rho_l}{\rho \rho_l}\tilde{w}_l^T}+ \norm{z_l-z}
\]
The second and third summands are clearly $\bigO((\frac{\abs{\lambda_2+\epsilon}}{\rho})^l)$ .
For the first summand, as $\nabla N_l^T u_{l}=1$, we have
\begin{align*}
\abs{z_l^T u_{l+1}}& =\abs{\frac{1}{\rho_l}(-\nabla f_l^T u_{l+1} (\nabla f_l^T u_l)(\nabla N_l^T u_{l+1}))}  \\
& \leq \frac{1}{\rho_l} (\sup_{l \geq 0} \norm{\nabla f^T_l} + \sup_{l \geq 0} \norm{\nabla N^T_l}) \norm{u_{l+1}-u_l} \enspace.
\end{align*}
As $(u_l)_{l\geq0}$ is bounded and $f$ and $N$ are $C^1$, the constant 
is finite and $\abs{z_l^T u_{l+1}}\norm{v_{l+1}} = \bigO ((\frac{\abs{\lambda_2+\epsilon}}{\rho})^l)$.
With similar arguments, we also show that $\norm{z_l-z} = \bigO ((\frac{\abs{\lambda_2+\epsilon}}{\rho})^l)$

Finally, we remark that for all $k$, $(\tilde{z}_{l-1-k}^T-z^T) (\tilde{M} (I-P))^{k}= 
\bigO ((\frac{\abs{\lambda_2+\epsilon}}{\rho})^{l-1})$.
Thus 
\[
\sum_{k=0}^{l-1} (\tilde{z}_{l-1-k}^T-z^T) (\tilde{M} (I-P))^{k} = \bigO (l(\frac{\abs{\lambda_2+\epsilon}}{\rho})^{l-1})= \bigO ((\frac{\abs{\lambda_2}+\epsilon'}{\rho})^l) \enspace,
\]
for all $\epsilon'>\epsilon$. The result follows.
\end{proof}

\begin{remark}
All this applies easily to a nonnegative irreducible matrix $M$. Let $\rho$ be its principal eigenvalue:
it is simple thanks to irreducibility. The spectral gap assumption $\rho>\abs{\lambda_2}$
 is guaranteed by an additionnal aperiodicity assumption. 
Let $v$ and $u$ be the left and right eigenvectors of $M$ for the eigenvalue $\rho$.
We normalize $u$ by $N(u) =1$ where $N$ verifies $\frac{\partial N}{\partial u}(u) \geq 0$ and $N(\lambda u)=\lambda N(u)$ 
(which implies $\frac{\partial N}{\partial u}(u) \cdot u =N(u)$) and $v$ by $\sum_i u_i v_i =1$.
As $u>0$, any normalization such that $p=\frac{\partial N}{\partial u}(u) \geq 0$ is satisfactory:
for instance, we could choose $N(u)=\norm{u}_1=\sum_i u_i$, $N(u)=\norm{u}_2$ or $N(u)=u_1$.
\end{remark}

\begin{remark} \label{rem:grad}
Theorem~\ref{thm:fastDer} gives the possibility of performing a gradient algorithm for
Perron vector optimization. Fix $\epsilon$ and apply recursively the power-type iterations $\mathrm{POWERDERIVATIVE_M}$
until $\norm{u_l - u_{l+1}}+\norm{\tilde{w}_l - \tilde{w}_{l+1}}\leq \epsilon$.
Then we use $\tilde{w}_l u_l$ as the descent direction of the algorithm.
The gradient algorithm will stop at a nearly stationary point,
the smaller $\epsilon$ the better.
In order to accelerate the algorithm, we can initialize the
recurrence with former values of $u_l$, $v_l$ and $\tilde{w}_l$.
\end{remark}

%\section{Projected approximate gradient algorithm} \label{sec:approxGradient}
\section{Coupling gradient and power iterations} \label{sec:approxGradient}

We have given in Theorem~\ref{thm:fastDer} an algorithm that gives the derivative of the objective function
at the same computational cost as the computation of the value of the function.
As the problem is a differentiable optimization problem,
we can perform any classical optimization algorithm: 
see~\cite{BGLS-numericalOpt, Bert-nonlinear, Nocedal-numericalOpt} for references.

When we consider relaxations of HITS autority or HOTS optimization problems,
that we define in Sections~\ref{sec:HITS} and~\ref{sec:HOTS},
the constraints on the adjacency matrices are very easy to deal with, so
a projected gradient algorithm as described in \cite{Ber-gradientAlg} will be efficient.
If the problem has not a too big size, it is also possible to set a second
order algorithm.
% : the formulas for the hessian matrix are similar to those
% for the derivative (see formulas in the Appendix).
However, matrices arising from web applications are large: as the Hessian matrix is a full $n^2 \times n^2$ matrix,
it is then difficult to work with.

In usual algorithms, the value of the objective function
must be evaluated at any step of the algorithm.
As stressed in~\cite{Overton-largeScale}, there are various possibilities for the computation
of the eigenvalue and eigenvectors. Here, we consider sparse nonnegative matrices with a simple principal eigenvalue:
the power method applies and, unlike direct methods or inverse iterations, it only needs matrix-vector products,
which is valuable with a large sparse matrix.
Nevertheless for large matrices, repeated principal eigenvector and eigenvalue
determinations can be costly.
Hence, we give a first order descent algorithm designed to find stationary points of 
Perron eigenvector optimization problems, that uses approximations of
the value of the objective and of its gradient instead of the precise values. Then we can
interrupt the computation of the eigenvector and eigenvalue when necessary
and avoid useless computations. Moreover, as the objective is evaluated as the limit
of a sequence, its exact value is not available in the present context.

The algorithm consists of a~coupling of the power iterations and of the gradient algorithm with Armijo
line search along the projected arc~\cite{Ber-gradientAlg}. We recall
this gradient algorithm in Algorithm~\ref{alg:Armijo}.
We shall, instead of comparing the exact values of the function, compare
upper and lower bounds computed during the course of the power iterations.

\begin{algorithm}
\caption{Gradient algorithm with Armijo line search along the projected arc~\cite{Ber-gradientAlg}}
Let a differentiable function $J$, a~convex admissible set $\C$
and an initial point $x_0 \in \C$ and parameters $\sigma \in (0,1)$, $\alpha^0>0$ and $\beta \in (0,1)$.
%Given $x \in C$, we denote $x(\alpha)=\mathrm{P}_\C(x-\alpha \nabla J(x) )$, where $\mathrm{P}_\C$ is the projection on the convex set $C$.
The algorithm is an iterative algorithm defined for all $k \in \mathbb{N}$ by
\[
 x_{k+1}=\mathrm{P}_\C(x_k-\alpha_k \nabla J(x_k) )
\]
and $\alpha_k=\beta^{m_k} \alpha^0$ where $m_k$ is the first nonnegative integer $m$ such that
\[
J\left (\mathrm{P}_\C(x_k- \beta^{m} \alpha^0 \nabla J(x_k))\right )- J(x_k) \leq
%-\sigma \langle \nabla J(x_k), x_k - \mathrm{P}_\C(x_k- \beta^{m} \alpha^0\nabla J(x_k))\rangle
-\sigma \frac{\norm{x_k - \mathrm{P}_\C(x_k- \beta^{m} \alpha^0\nabla J(x_k))}_2^2}{\beta^{m} \alpha^0}
\]
\label{alg:Armijo}
\end{algorithm}

If we had an easy access to the exact value of $u(M)$ for all $M \in h(\C)$, we could use
the gradient algorithm with Armijo line search along the projected arc with $J=f \circ u \circ h$
to find a stationary point of Problem~\eqref{eq:perronVecOpt}.
But when computing the Perron eigenvector by an iterative scheme like in Theorem~\ref{thm:fastDer},
we only have converging approximations of the value of the objective and of its gradient.
The theory of consistent approximation, developped in~\cite{Polak-consistent} proposes algorithms and
convergence results for such problems. If the main applications of consistent approximations
are optimal control and optimal control of partial derivative equations, it is also useful
for problems in finite dimension where the objective is difficult to compute~\cite{Pironneau-ApproxGrad}.

A consistent approximation of a given optimization problem is
a sequence of computationally tractable problems that converge
to the initial problem in the sense that the stationary points
of the approximate problems
converge to stationary points of the original problem.
The theory provides master algorithms that construct a consistent approximation,
initialize a nonlinear programming algorithm on this approximation and
terminate its operation when a precision (or discretization) improvement test is satisfied.

We consider the Perron vector optimization problem defined in~\eqref{eq:perronVecOpt}
\begin{equation*}
 \min_{x \in \C} J(x)= \min_{x \in \C} f \circ u \circ h(x) \enspace .
\end{equation*}
For $x \in \C$, $n\in \mathbb{N}$ and for arbitrary fixed vectors $u_0$, $v_0$ and $\tilde{w}_0$, we shall approximate 
with order $\Delta(n)$
the Perron vectors of $h(x)$, namely $u$ and $v$ and the auxiliary vector $w$ of Corollary~\ref{prop:gcool}
by 
\begin{equation} \label{eq:powerderivativen}
 (u_{k_n}, v_{k_n}, \tilde{w}_{k_n}):=(\mathrm{POWERDERIVATIVE_{h(x)}})^{k_n}(u_0, v_0, \tilde{w}_0) \enspace ,
\end{equation}
where $k_n$ is the first nonnegative integer $k$ such that
\begin{equation} \label{eq:kn}
 \norm{(u_{k+1}, v_{k+1}, \tilde{w}_{k+1}) - (u_{k}, v_{k}, \tilde{w}_{k})} \leq \Delta(n)
\end{equation}

The map $\mathrm{POWERDERIVATIVE}$ is defined in Theorem~\ref{thm:fastDer}.
Then the degree $n$ approximation of the objective function $J$ and of its gradient $\nabla J$ are given by
\begin{align} \label{eq:approx}
 J_n(x)=f(u_{k_n})\;,  \qquad g_n(x)=\sum_{i,j \in [n]} \tilde{w}_{k_n}(i) \nabla h_{i,j}(x) u_{k_n}(j)
\end{align}

An alternative approach, proposed in~\cite{Pironneau-ApproxGrad}, is to approximate $(u, v, w)$ by 
the $n^{th}$ iterate $(u_n, v_n, \tilde{w}_n):=(\mathrm{POWERDERIVATIVE_{h(x)}})^{n}(u_0, v_0, \tilde{w}_0)$.
We did not choose this approach since it does not take into account efficiently hot started power iterations.

We define an approximate gradient step with Armijo line search in Algorithm~\ref{alg:ArmijoApprox}.

\begin{algorithm}
\caption{Approximate Armijo line search along the projected arc}
Let $(\bar{M}_n)_{n\geq0}$ be a~sequence diverging to $+\infty$, 
$\sigma \in (0,1)$, $\alpha^0>0$, $\beta \in (0,1)$ and $\gamma>0$.
Given $n \in \mathbb{N}$, $J_n$ and $g_n$ are defined in~\eqref{eq:approx}.
For $x \in \C$,
the algorithm returns $A_n(x)$ defined as follows.
If for all nonnegative integer $m$ smaller than $\bar{M}_n$,
\[
J_n\left (\mathrm{P}_\C(x- \beta^{m} \alpha^0 g_n (x))\right )- J_n(x) >
-\sigma \frac{\norm{x - \mathrm{P}_\C(x- \beta^{m} \alpha^0 g_n (x))}_2^2}{\beta^{m} \alpha^0}
\]
then we say that the line search has failed and we set $A_n(x)=\emptyset$.
Otherwise, 
let $m_n$ be the first nonnegative integer $m$ such that
\[
 J_n\left (\mathrm{P}_\C(x- \beta^{m} \alpha^0 g_n (x))\right ) -J_n(x)\leq
-\sigma \frac{\norm{x - \mathrm{P}_\C(x- \beta^{m} \alpha^0 g_n (x))}_2^2}{\beta^{m} \alpha^0}
\]
and define the next iterate $A_n(x)$ to be $A_n(x)=\mathrm{P}_\C(x- \beta^{m_n} \alpha^0 g_n(x) )$.
\label{alg:ArmijoApprox}
\end{algorithm}

% The presence of $\gamma \Delta(n)>0$ is necessary since $g_n$ may not be a descent direction,
% so that $m_n$ may not exist if $\gamma=0$.
Then we shall use the Approximate Armijo line search along the projected arc $A_n$
in the following Master Algorithm Model (Algorithm~\ref{alg:Master}).

\begin{algorithm}
\caption{Master Algorithm Model 3.3.17 in \cite{Polak-consistent}}
Let $\omega\in (0,1)$, $\sigma' \in (0,1)$, $n_{-1} \in \mathbb{N}$ and $x_0 \in \C$,
$\mathcal{N}= \{ n \; | \; A_{n}(x) \not= \emptyset\;, \; \forall x \in C \}$ 
and $(\Delta(n))_{n\geq0}$ be a~sequence converging to $0$.

For $i \in \mathbb{N}$, compute iteratively the smallest $n_i \in \mathcal{N}$ and $x_{i+1}$ such that $n_i \geq n_{i-1}$,
 \begin{align*}
  x_{i+1} \in A_{n_i}(x_i) \qquad \text{and} \\
 J_{n_i}(x_{i+1}) - J_{n_i}(x_i) \leq - \sigma' \Delta(n_i)^{\omega}
\end{align*}
\label{alg:Master}
\end{algorithm}

In order to prove the convergence of the Master Algorithm Model (algorithm~\ref{alg:Master}) 
when used with the Approximate Armijo line search
(Algorithm~\ref{alg:ArmijoApprox}), we need the following lemma.
\begin{lemma} \label{lem:ArmijoApproxGood}
For all $x^* \in \C$ which is not stationary, there exists $\rho^*>0$, $\delta^*>0$ and $n^* \in \mathbb{N}$ such that for all $n \geq n^*$,
and for all $x \in \mathrm{B}(x^*, \rho^*) \cap \C$, $A_n(x) \not = \emptyset$ and
\[
 J_n(\left (\mathrm{P}_\C(x- \alpha_n g_n(x))\right )- J_n(x) \leq - \delta^*
\]
where $\alpha_n$ is the step length returned by the Approximate Armijo line search $A_n(x)$ (Algorithm~\ref{alg:ArmijoApprox}).
\end{lemma}
\begin{proof}
 Let $x\in \C$. Suppose that there exists an
infinitely growing sequence $(\phi_n)_{n \geq 0}$ such that
$A_{\phi_n}(x) = \emptyset$ for all $n$.
Then for all $m \leq \bar{M}_{\phi_n}$,
\[
J_{\phi_n}\left (\mathrm{P}_\C(x- \beta^{m} \alpha^0 g_{\phi_n} (x))\right )- J_{\phi_n}(x) >
% -\sigma \frac{\norm{x - \mathrm{P}_\C(x- \beta^{m} \alpha^0 g_{\phi_n} (x))}_2^2}{\beta^{m} \alpha^0} \enspace .
 -\sigma  \langle g_{\phi_n} (x), x - \mathrm{P}_\C(x- \beta^{m} \alpha^0 g_{\phi_n} (x)) \rangle
\]
When $n\to +\infty$, $\bar{M}_{\phi_n} \to +\infty$, $J_{\phi_n}(x) \to J(x)$ and $g_{\phi_n} (x) \to \nabla f (x)$ (Theorem~\ref{thm:fastDer}),
so we get that for all $m \in \mathbb{N}$, $J\left (\mathrm{P}_\C(x_k- \beta^{m} \alpha^0 \nabla J(x_k))\right )- J(x_k) \geq
 -\sigma \frac{\norm{x_k - \mathrm{P}_\C(x_k- \beta^{m} \alpha^0\nabla J(x_k))}_2^2}{\beta^{m} \alpha^0}$,
which is impossible by~\cite{Ber-gradientAlg}.

So suppose $n \in \mathbb{N}$ is sufficiently large so that $A_{_n}(x) \not = \emptyset$.
Let $\alpha_n$ the step length determined by Algorithm~\ref{alg:ArmijoApprox} 
and let $\alpha$ be the step length determined by Algorithm~\ref{alg:Armijo} at $x$.
We have:
\begin{align*}
 J_n\left (\mathrm{P}_\C(x- \alpha_n g_n (x))\right ) -J_n(x)\leq
 -\sigma \frac{\norm{x - \mathrm{P}_\C(x- \alpha_n g_n (x))}_2^2}{\alpha_n}
\end{align*}
and if $\alpha_n \not = \alpha^0$,
\begin{align*}
J_n\left (\mathrm{P}_\C(x- \beta^{-1} \alpha_n g_n(x))\right )- J_n(x) >
 -\sigma \frac{\norm{x - \mathrm{P}_\C(x- \beta^{-1} \alpha_n g_n(x))}_2^2}{\beta^{-1} \alpha_n}
\end{align*}
$(\alpha_n)_{n \geq 0}$ is a bounded sequence so it has a subsequence $(\alpha_{\phi_n})_{n \geq 0}$ converging to, say, $\bar{\alpha}$.
As $(\alpha_n)_{n \geq 0}$ can only take discrete values, this means that
$\alpha_{\phi_n} = \bar{\alpha}$ for all $n$ sufficiently big.

When $n$ tend to infinity, by Theorem~\ref{thm:fastDer}, we get 
\begin{align*}
 J\left (\mathrm{P}_\C(x- \bar{\alpha} \nabla f (x))\right ) -J(x)\leq
 -\sigma \frac{\norm{x - \mathrm{P}_\C(x- \bar{\alpha} \nabla f (x))}_2^2}{\bar{\alpha}}
\end{align*}
and if $\bar{\alpha} \not = \alpha^0$,
\begin{align*}
J\left (\mathrm{P}_\C(x- \beta^{-1} \bar{\alpha} \nabla f (x))\right )- J(x) \geq
 -\sigma \frac{\norm{x - \mathrm{P}_\C(x- \beta^{-1}  \bar{\alpha} \nabla f(x))}_2^2}{\beta^{-1} \bar{\alpha}}
\end{align*}
Then, if $\alpha$ is the step length returned by Armijo rule (Algorithm~\ref{alg:Armijo}),
then $\alpha \geq \bar{\alpha}$, because $\alpha$ is the first number of the sequence that
verifies the first inequality. Similarly, consider the version
of Armijo rule with a strict inequality instead of the non strict inequality.
Then if $\alpha_{\mathrm{strict}}$ is the step length returned by this algorithm,
we have $\alpha_{\mathrm{strict}} \leq \bar{\alpha}$. 

Moreover, like in~\cite{Pironneau-ApproxGrad} one can easily see that
$\forall x^* \in \C$ not stationary,  $\exists \rho^*>0$, $\exists \delta^*>0$, such that
$\forall x \in \mathrm{B}(x^*, \rho^*) \cap \C$,
\[
 J(\left (\mathrm{P}_\C(x- \alpha_{\mathrm{strict}} \nabla f(x))\right )- J(x) \leq - \delta^*
\]
where $\delta^* = \sigma \frac{\norm{x^* - \mathrm{P}_\C(x^*- \alpha_{\mathrm{strict}} \nabla f (x^*))}_2^2}{\alpha_{\mathrm{strict}}} - \norm{\nabla J(x^*)} \rho^*$.
% \begin{align}
% J\left (\mathrm{P}_\C(x_k- \alpha \nabla J(x_k))\right )- J(x_k) \leq
%  -\sigma \frac{\norm{x_k - \mathrm{P}_\C(x_k- \alpha\nabla J(x_k))}_2^2}{\alpha} \\
% J\left (\mathrm{P}_\C(x_k- \beta^{-1} \alpha \nabla J(x_k))\right )- J(x_k) \leq
%  -\sigma \frac{\norm{x_k - \mathrm{P}_\C(x_k- \beta^{-1} \alpha\nabla J(x_k))}_2^2}{\beta^{-1} \alpha}
% \end{align}
By Lemma~3 in \cite{Ber-gradientAlgImproved}, $\alpha_{\mathrm{strict}} \leq \bar{\alpha}$ implies that
$\frac{\norm{x^* - \mathrm{P}_\C(x^*- \alpha_{\mathrm{strict}} \nabla f (x^*))}_2^2}{\alpha_{\mathrm{strict}}} \geq
 \frac{\norm{x - \mathrm{P}_\C(x- \bar{\alpha} \nabla f(x))}_2^2}{\bar{\alpha}}$.
As this is true for all adherent point of $(\alpha_n)_{n\geq 0}$,
$\forall x^* \in \C$ not stationary,  $\exists \rho^*>0$, $\exists \delta^*>0$ and $\exists n^* \in \mathbb{N}$, such that
$\forall n \geq n^*$, $\forall x \in \mathrm{B}(x^*, \rho^*) \cap \C$,
\begin{align*}
 J_n(\left (\mathrm{P}_\C(x- \alpha_n g_n(x))\right )- J_n(x) &\leq 
 -\sigma \frac{\norm{x - \mathrm{P}_\C(x- \alpha_{\mathrm{strict}} \nabla f (x))}_2^2}{\alpha_{\mathrm{strict}}} + \delta^*/4 % \\
%&
\leq - \delta^*/2 \enspace ,
\end{align*}
for $n^*$ sufficiently large and $\rho^*$ sufficiently small. 
\end{proof}
In~\cite{Pironneau-ApproxGrad}, the property of the lemma was proved for exact minimization
in the line search. We proved it for Algorithm~\ref{alg:ArmijoApprox}.

We shall also need the following result
\begin{prop}[Theorem 25 in \cite{mayer-eigenVerification}] \label{prop:boundVecNonsym}
 Let $M \in M_{n,n}(\mathbb{R})$, $\tilde{\lambda} \in \mathbb{R}$,
$\tilde{x} \in \mathbb{R}^n$, $C \in M_{n+1,n+1}(\mathbb{R})$ and $p \in \mathbb{R}^n$ such that $p^T \tilde{x}=1$.
Denote 
\[
B= \begin{bmatrix}A-\tilde{\lambda}I_n & -\tilde{x} \\ p^T  & 0 \end{bmatrix} \enspace , \quad 
\eta= \left \lVert C \begin{bmatrix}A \tilde{x}-\tilde{\lambda}\tilde{x} \\ 0\end{bmatrix} \right \lVert_{\infty}\enspace ,
\] 
$\sigma=\norm{I_{n+1}-CB}$ and $\tau=\norm{C}_{\infty}$.
If $\sigma <1$ and $\Delta=(1-\sigma)^2-4\eta \tau \geq 0$, then
$\beta=\frac{2 \eta}{1-\sigma + \sqrt{\Delta}}$ is nonnegative and
there exists a unique eigenpair $(x^*,\lambda^*)$ of $M$ such
that $p^T x^*=1$, $\abs{\lambda^*-\tilde{\lambda}}\leq \beta$ and 
 $\norm{x^*-\tilde{x}}_{\infty} \leq \beta$.
\end{prop}

\begin{thm} \label{thm:approxGrad}
 Let $(x_i)_{i\geq 0}$ be a sequence constructed by the Master Algorithm Model (Algorithm~\ref{alg:Master})
for the resolution of the Perron vector optimization problem~\eqref{eq:perronVecOpt} 
such that $A_{n}(x)$ is the Approximate Armijo line search along the projected arc
(Algorithm~\ref{alg:ArmijoApprox})
and $\Delta(n)=(\Delta_0)^n$ for $\Delta_0 \in (0,1)$.
Then every accumulation point of $(x_i)_{i\geq 0}$ is
a stationary point of~\eqref{eq:perronVecOpt}.
\end{thm}

\begin{proof}
The proof of the theorem is based on Theorem~3.3.19 in~\cite{Polak-consistent}. 
This theorem shows that if continuity assumptions hold (they trivially hold in our case), if for all bounded subset $S$ of $\C$ there exist $K>0$
such that for all $x\in \C$
\[
\abs{J(x)-J_n(x)} \leq K \Delta(n) \enspace,
\]
and if for all $x^* \in \C$ which is not stationary, there exists $\rho^*>0$, $\delta^*>0$ and $n^* \in \mathbb{N}$ such that for all $n \geq n^*$,
for all $x \in \mathrm{B}(x^*, \rho^*) \cap \C$ and for all $y\in A_n(x)$,
\[
 J_n(y)- J_n(x) \leq - \delta^* \enspace ,
\]
then every accumulation point of a sequence $(x_i)_{i\geq 0}$ generated by the
Master Algorithm Model (algorithm~\ref{alg:Master}) is
a stationary point of the problem of minimizing $J(x)$.

We first remark that for $x\in \C$, $u=u(h(x))$ and $J_n$~defined in~\eqref{eq:approx}, 
as $f$~is continuously differentiable,
 we have $\abs{J(x)-J_n(x)} \leq \norm{\nabla f(u)} \; \norm{u-u_{k_n}}$.

We shall now show that for all matrix $M$,
there exists $K>0$ such that $\norm{u-u_n} \leq K \norm{u_{n+1}-u_{n}}$.
Remark that Theorem 16 in \cite{mayer-eigenVerification} gives the result with $K>\frac{N(u_n)}{\rho-\abs{\lambda_2}}$ when $M$ is symmetric.
When $M$ is not necessarily symmetric, we use Proposition~\ref{prop:boundVecNonsym}
with 
% \[
% C=\begin{bmatrix} (I - u \nabla N^T(x))(M-\rho I)^{\#} & u \\ -v & 0 \end{bmatrix} \enspace.
% \]
 $\tilde{x}=u$, $\tilde{\lambda}=\rho$ and $C$~is the inverse of~$B$. 
% The largest eigenvalue of $(M-\rho I)^{\#}$ is $\frac{1}{\abs{\lambda_2(M)}-\rho(M)}$,
% so there exists $K>0$ such that $\tau:=\norm{C}_{\infty} \leq \frac{K}{\rho-\abs{\lambda_2}}$. 
Let $\epsilon>0$, by continuity of the inverse, for $n$ sufficiently large,
if we define $B_n$ to be the matrix of Proposition~\ref{prop:boundVecNonsym} with
$\tilde{x}=u_n$ and $\tilde{\lambda}=N(M u_n)$,
we still have $\sigma:=\norm{I_{n+1}-CB_n} < \epsilon$.
We also have: 
\begin{align*}
 \eta:=\left \lVert C \begin{bmatrix} M \tilde{x}-\tilde{\lambda}\tilde{x} \\ 0\end{bmatrix} \right \lVert_{\infty} 
&\leq \norm{C}_{\infty} \norm{M u_n- N(M u_n)u_n}_{\infty} \\
&\leq \norm{C}_{\infty} N(M u_n) \norm{u_{n+1} - u_n} \enspace .
\end{align*}
The conclusion of Proposition~\ref{prop:boundVecNonsym} tells us that if
$\Delta := (1-\sigma)^2-4 \eta \tau \geq 0$, then
\[
 \norm{u-u_n} \leq \beta := \frac{2 \eta}{1-\sigma + \sqrt{\Delta}} \leq \frac{2 \eta}{1-\sigma} 
\leq 3 \eta \leq 3 \norm{C}_{\infty} N(M u_n) \norm{u_{n+1} - u_n} \enspace .
\]
Now, as the inversion is a continuous operation, 
for all compact subset $S$ of $\C$,
there exists $K>0$ such that for all $M=h(x) \in h(\C)$ and for all $n$ sufficiently big, $\norm{u-u_n} \leq K \norm{u_{n+1}-u_{n}}$.
% As the power method converges with linear rate of convergence $\frac{\abs{\lambda_2}}{\rho}$ \cite{ParlettPoole-power},
% there exists $K>0$ and $\frac{\abs{\lambda_2}}{\rho}<s<1$ such that

By definition of $k_n$~\eqref{eq:kn}, we get
\[
\abs{J(x)-J_n(x)} \leq \norm{\nabla f(u)} \norm{u-u_{k_n}} \leq K \norm{u_{k_n+1}-u_{k_n}} \leq K \Delta(n) \enspace .
\]
% Similarly, suppose that $\tilde{w}_n - \tilde{w}_{n-1} \leq \Delta(n)$.
% By the proof of Theorem~\ref{thm:fastDer} there exists $\tilde{z}_{n-1}$ such that
% $\tilde{w}_n=-\tilde{z}_{n-1} + z -z + \tilde{w}_{n-1} \tilde{M} (I-uv^T)$, where
% $\tilde{z}_{n-1} - z \leq K \norm{u-u_n}$ for a $K>0$ and 

By Lemma~\ref{lem:ArmijoApproxGood}, the other hypothesis of Theorem~3.3.19 in~\cite{Polak-consistent}
is verified and the conclusion holds: every accumulation point of $(x_i)_{i\geq 0}$ is
a stationary point of~\eqref{eq:perronVecOpt}.
\end{proof}

\section{Application to HITS optimization} \label{sec:HITS}

In the last two sections, we have developped scalable algorithms for the computation
of the derivative of a scalar function of the Perron vector of a matrix and for
the searching of stationary points of Perron vector optimization problems~\eqref{eq:perronVecOpt}.
We now apply these results to two web ranking optimization problems, namely
HITS authority optimization and HOTS optimization.

HITS algorithm for ranking web pages has been described by Kleinberg in~\cite{Kleinberg-HITS}.
The algorithm has two phases: first, given a query, it produces a subgraph $G$ of the whole web graph
such that in contains relevant pages, pages linked to relevant pages and the hyperlinks between them.
The second phase consists in computing a principal eigenvector called authority vector and to sort the pages
with respect to their corresponding entry in the eigenvector. If we denote by $A$ the adjacency
matrix of the directed graph $G$, then the authority vector is the principal eigenvector of $A^T A$.

It may however happen that $A^T A$ is reducible and then the authority vector is
not uniquely defined. Following~\cite{LanMey-Beyond}, we remedy this by defining the HITS authority score
to be the principal eigenvector of $A^T A + \xi e e^T$, for a given small positive real $\xi$.
We then normalize the HITS vector with the 2-norm as proposed be Kleinberg~\cite{Kleinberg-HITS}.

Given a subgraph associated to a query, we study in this section the optimization of
the authority of a set of pages.
We partition the set of potential links $(i,j)$ into three subsets, consisting respectively of the set of obligatory links~$\mathcal{O}$, the set of
prohibited links~$\mathcal{I}$ and the set of facultative links~$\mathcal{F}$.
Some authors consider that links between pages of a website, called intra-links,
should not be considered in the computation of HITS. This results in considering
these links as prohibited because this is as if they did not exist.

Then, we must select the subset $J$ of the set of
facultative links~$\mathcal{F}$ which are effectively included
in this page. Once this choice is made for every page, we get
a new webgraph, and define the adjacency matrix $A=A(J)$.
We make the simplificating assumption that the construction of the focused graph $G$ is
independent of the set of facultative links chosen.

Given a utility function $f$, the HITS authority optimization problem is:
\begin{equation} \label{eqn:HITSoptim}
  \max_{J \subseteq \mathcal{F},u \in \mathbb{R}^n,\lambda \in \mathbb{R}} \{ f(u) \; \; ; \; (A(J)^T A(J) +\xi e e^T) u =\lambda u \;, \; \norm{u}_2=1 \;, \; u \geq 0 
 \} 
\end{equation} 

The set of admissible adjacency matrices is a combinatorial set with a number of matrices exponential in
the number of facultative links. Thus we shall consider instead a relaxed version of the HITS authority optimization problem
which consists in accepting weighted adjacency matrices. It can be written as
\begin{equation} \label{eqn:HITSoptimRelaxed}
\begin{split}
  \max_{A\in \mathbb{R}^{n \times n},u \in \mathbb{R}^n,\rho \in \mathbb{R}} & f(u) \\
(A^T A +\xi e e^T) u =\rho u \;&, \;\ \norm{u}_2=1 \;, \; u \geq 0  \\
 A_{i,j}=1 \; &,\; \forall (i,j) \in \mathcal{O} \\
 A_{i,j}=0\; &,\; \forall (i,j) \in \mathcal{I} \\
0\leq A_{i,j} \leq 1\;&,\;  \forall (i,j) \in \mathcal{F}
\end{split}
\end{equation} 

The relaxed HITS authority optimization problem~\eqref{eqn:HITSoptimRelaxed}
is a Perron vector optimization problem~\eqref{eq:perronVecOpt}
with $h(A)=A^T A + \xi e e^T$ and the normalization $N(u)=\sqrt{\sum_i u_i^2}=1$.
Hence $\nabla N(u(M))= u(M)$.
Remark that $\norm{u}=\norm{v}=1$.
Now $\frac{\partial h}{\partial A}(A).H = H^T A + A^T H$
so the derivative of the criterion with respect to the weighted adjacency matrix is
$(Aw)u^T+(Au) w^T$ with $w=(\nabla f^T - (\nabla f \cdot u) \nabla N^T) (A^T A +\xi e e^T-\rho I)^{\#}$.

Thanks to $\xi>0$, the matrix is irredutible and aperiodic. Thus, it has only one eigenvalue
of maximal modulus and we can apply Theorem~\ref{thm:approxGrad}.

The next proposition shows that, as is the case for PageRank optimization~\cite{NinKer-PRopt, Fercoq-PRopt},
 optimal strategies have a rather simple structure.
\begin{prop}[Threshold property] \label{prop:shapeHits}
 Let $A$ be a locally maximal linking strategy of
 the relaxed HITS authority optimization problem~\eqref{eqn:HITSoptimRelaxed}
with associated authority vector $u$ and derivative at optimum $(Aw)u^T+(Au) w^T$.
For all controlled page~$i$ denote $b_i=\frac{-(Aw)_i}{(Au)_i}$ if it has at least one
outlink.
Then all facultative hyperlinks $(i,j)$ such that $\frac{w_j}{u_j}>b_i$  get a weight of~1 and
 all facultative hyperlinks $(i,j)$ such that $\frac{w_j}{u_j}<b_i$ get a weight of~0.

In particular, if two pages with different $b_i$'s have the same sets of facultative outlinks, then
their set of activated outlinks are included one in the other.
\end{prop}
\begin{proof}
 As the problem only has box constraints, 
a nonzero value of the derivative at the maximum
determines whether the upper bound is saturated ($g_{i,j}<0$) or the lower bound is
saturated ($g_{i,j}>0$). If the derivative is zero, the weight of the link can take
any value.

We have $g_{i,j}=(Aw)_i u_j+(Au)_i w_j$ with $u_j>0$ and $(Au)_i \geq 0$.
If Page~$i$ has at least one outlink, then $(Au)_i>0$ and we simply divide 
by $(Au)_i$ to get the result thanks to the first part of the proof.
If two pages~$i_1$ and~$i_2$ have the same sets of facultative outlinks
and if $b_{i_1}<b_{i_2}$, then $\frac{w_j}{u_j} \geq b_{i_2}$ implies $\frac{w_j}{u_j}>b_{i_1}$ all the pages pointed by $i_2$ are also pointed
by $i_1$.
\end{proof}
\begin{remark}
 If a page~$i$ has no outlink, then $(Aw)_i=(Au)_i=0$ and $g_{i,j}=0$ for all 
$j\in [n]$, so we cannot conclude with the argument of the proof.
\end{remark}
\begin{remark}
This proposition shows that $\frac{w_j}{u_j}$ gives a total order of preference
in pointing to a page or another.
\end{remark}

Then we give on Figures~\ref{fig:GraphOptHits1} and~\ref{fig:GraphOptHits2} 
a simple HITS authority optimization problem and two local solutions.
They show the following properties for this problem.

\begin{figure} 
\centering
\small{
\begin{tikzpicture}[>=latex', scale=0.26]
\pgfsetlinewidth{0.5bp}
\input{GraphOptHits1.tex} 
\end{tikzpicture}
}
\caption{Strict local maximum for relaxed HITS authority optimization on 
a~small web graph of 21~pages with 3~controlled pages (colored) representing the website $I$. Obligatory links are
the thin arcs, facultative links are all the other outlinks from the controlled pages except self links.
The locally optimal solution for the maximization of $f(u)=\sum_{i\in I} u_i^2$ is to select the bold arcs with weight~$1$ and the dotted arc with weight~$0.18$.
Selected internal links are dark blue, selected external links are light red.
We checked numerically the second order optimality conditions~\cite{Bert-nonlinear}.
%We computed the Hessian matrix restricted to non saturated directions by numerical differentiation.
}\label{fig:GraphOptHits1}
\end{figure}
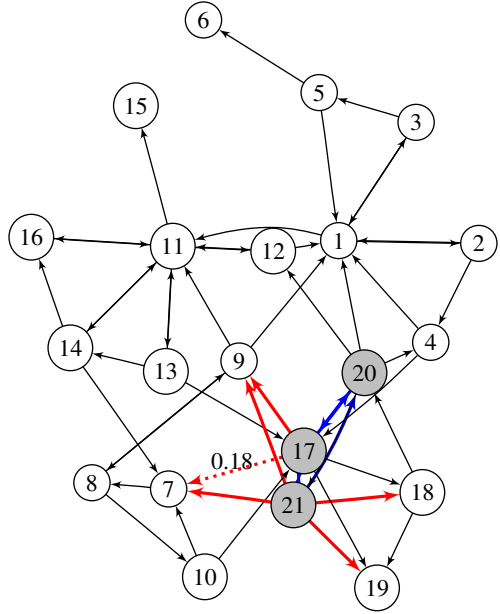

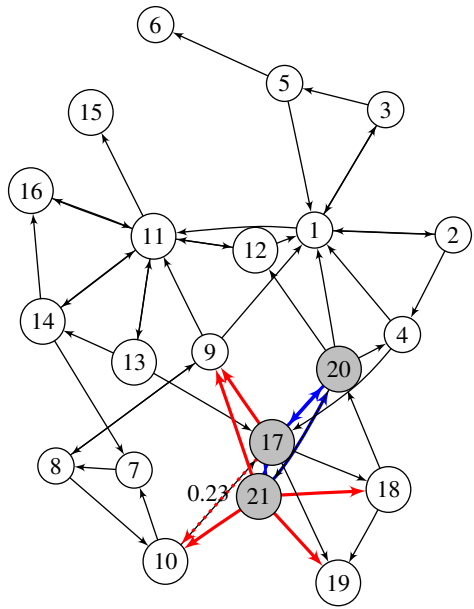
\begin{figure}
\centering
\small{
\begin{tikzpicture}[>=latex', join=bevel, scale=0.26]
\pgfsetlinewidth{0.5bp}
\input{GraphOptHits2.tex} 
\end{tikzpicture}
}
\caption{Another strict local maximum for the same HITS authority optimization problem as in Figure~\ref{fig:GraphOptHits1}.}
\label{fig:GraphOptHits2}
\end{figure}

\begin{countex} \label{prop:HitsNotQuasiConvexConcave}
The relaxed HITS authority optimization problem is in general not quasi-convex nor quasi-concave.
\end{countex}
\begin{proof}
 Any strict local maximum of a quasi-convex problem is necessarily an extreme of the admissible polyhedral set
(this is a simple extension of Theorem~3.5.3 in~\cite{Bazaraa-Nonlinopt}).
The example on Figure~\ref{fig:GraphOptHits1} shows that this is not the case here.

A quasi-concave problem can have only one strict local maximum (although it may have many local maxima).
The examples on Figures~\ref{fig:GraphOptHits1} and~\ref{fig:GraphOptHits2}
show two distinct strict local maxima for a HITS authority optimization problem.
\end{proof}

\begin{heuristic}
These examples also show that the relaxed HITS authority optimization problem~\eqref{eqn:HITSoptim} does not give
binary solutions that would be then solutions of the initial HITS authority optimization problem~\eqref{eqn:HITSoptimRelaxed}.
Hence we propose the following heuristic to get ``good'' binary solutions.
From a stationary point of the relaxed problem, define the function $\phi : [0,1] \to \mathbb{R}$
such that $\phi(x)$ is the value of the objective function when we select in~\eqref{eqn:HITSoptim} all the 
links with weight bigger than $x$ in the stationary point of~\eqref{eqn:HITSoptimRelaxed}.
We only need to compute it at a finite number of points since $\phi$ is piecewise constant.
We then select the best threshold.
For instance, with the stationary point of Figure~\ref{fig:GraphOptHits1}, this heuristic suggests not to
select the weighted link.
\end{heuristic}

\section{Optimization of HOTS} \label{sec:HOTS}

\subsection{Tomlin's HOTS algorithm}

HOTS algorithm was introduced by Tomlin in~\cite{Tomlin-HOTS}.
In this case, the ranking is the vector of dual variables of an optimal flow problem. The flow
represents an optimal distribution of web surfers on the web graph in the sense of
entropy minimization. It only takes into account the web graph, so
like PageRank, it is query independent link-based search algorithm. 
The dual variable, one by page, is interpreted as the ``temperature'' of
the page, the hotter a page the better.
Tomlin showed that this vector is solution of a nonlinear fix point equation: it may be seen
as a nonlinear eigenvector. Indeed, most of the arguments available in the case
of Perron vector optimization can be adapted to HOTS optimization: we show that
the matrix of partial derivatives of the objective has a low rank and
that it can be evaluated by a power-type algorithm. Also, as for PageRank and HITS authority optimization,
we show that a threshold property holds.

Denote $G=(V,E)$ the web graph with adjacency matrix $A=(A_{i,j})$ and consider a modified graph
$G'=(V',E')$ where $V'=V\cup \{n+1\}$ and $E'=E \cup (\cup_{i\in V} \{i, n+1\}) \cup (\cup_{j \in V} \{n+1, j\})$.
Fix $\alpha \in (1/2,1)$.
Then the maximum entropy flow problem considered in~\cite{Tomlin-HOTS} is given by
\begin{align*}
\max_{\rho \geq 0} & -\sum_{ e \in E'} \rho_e \log(\rho_e)  \\
\sum_{j \in V'} \rho_{i,j} = &\sum_{j \in V'} \rho_{j,i} \; , \; \forall i \in V'  & (p_i) \\
\sum_{i,j \in V'} \rho_{ij} = & 1        & (\mu) \\
\sum_{j \in V'} \rho_{n+1,j} = & 1-\alpha & (a_{n+1}) \\
1-\alpha =& \sum_{i \in V'} \rho_{i,n+1}&  (b_{n+1})
\end{align*}

The dual of this optimization problem is
\begin{equation} \label{eqn:dualHots}
 \begin{split}
   \min_{(p, \mu, a_{n+1}, b_{n+1}) \in \mathbb{R}^n \times \mathbb{R}^3} \tilde{\theta}(p,\mu,a_{n+1},b_{n+1}) :=  \sum_{i,j \in [n]} A_{ij} e^{p_i - p_j+\mu} + \sum_{i \in [n]} e^{-b_{n+1}+p_i+\mu} \\
+ \sum_{j \in [n]} e^{a_{n+1}-p_j +\mu} -(1-\alpha)a_{n+1}-\mu+(1-\alpha)b_{n+1} \enspace .
 \end{split}
\end{equation}

This problem is a form of matrix balancing (see~\cite{Schneider-scaling}).
For $p$ being an optimum of Problem~\eqref{eqn:dualHots}, 
the \NEW{HOTS value} of page~$i$ is defined to be $\exp(p_i)$.
 It is interpreted as the temperature of the page and we sort pages
from hotter to cooler.

The dual problem~\eqref{eqn:dualHots} consists in the non constrained minimization of a convex function,
thus a necessary and sufficient optimality condition is the cancelation of the gradient.
From these equations, we can recover a dual form of the fix point equation described in~\cite{Tomlin-HOTS}. Note that we take the convention of~\cite{Schneider-scaling}
so that our dual variables are the opposite of Tomlin's.

From the expressions of $a_{n+1}$, $b_{n+1}$ and $\mu$ at the optimum, respectively given by $e^{a_{n+1}} = \frac{1-\alpha}{\sum_{j \in [n]}e^{-p_j}}e^{-\mu}$, $e^{-b_{n+1}} = \frac{1-\alpha}{\sum_{i \in [n]}e^{p_i}}e^{-\mu}$ and $e^{\mu}=\frac{2\alpha-1}{\sum_{i,j \in [n]} A_{ij} e^{p_i - p_j}}$,
we can write $\tilde{\theta}$ as a function of $p$ only:
\begin{equation} \label{eqn:theta}
\begin{split}
 \theta(p) & = \min_{a_{n+1}, b_{n+1},\mu} \tilde{\theta}(p,\mu,a_{n+1},b_{n+1})\\
&=C(\alpha)+\phi(-p) +\phi(p) + (2\alpha -1) \log(\sum_{i,j\in [n]} A_{i,j} e^{p_i-p_j})
\end{split}
\end{equation}
where $C(\alpha)=1 -2(1-\alpha) \log(1-\alpha)-(2\alpha-1) \log(2\alpha-1)$ and where $\phi(p) = (1-\alpha) \log(\sum_{i\in [n]} e^{p_i})$. Its gradient is given by
%\begin{multline*}
\begin{equation*}
 \frac{\partial \theta}{\partial p_l}(p)=-(1-\alpha) \frac{e^{-p_l}}{\sum_j e^{-p_j}} +(1-\alpha) \frac{e^{p_l}}{\sum_i e^{p_i}} % \\
-(2\alpha-1) \frac{\sum_i A_{il}e^{p_i-p_l}}{\sum_{i,j} A_{ij} e^{p_i-p_j}} + (2\alpha-1) \frac{\sum_j A_{lj}e^{p_l-p_j}}{\sum_{i,j} A_{ij} e^{p_i-p_j}}
\end{equation*}
%\end{multline*}
This equality can be also written as
%\begin{multline*}
\begin{equation*}
e^{2 p_l} \left( (2\alpha-1) \frac{\sum_j A_{lj}e^{-p_j}}{\sum_{i,j} A_{ij} e^{p_i-p_j}} +\frac{1-\alpha}{\sum_i e^{p_i}} \right ) = %\\
(2\alpha-1) \frac{\sum_i A_{il}e^{p_i}}{\sum_{i,j} A_{ij} e^{p_i-p_j}} + \frac{1-\alpha}{\sum_j e^{-p_j}} + e^{p_l}\frac{\partial \theta}{\partial p_l}(p)
\end{equation*}
%\end{multline*}
which yields for all $p$ in $\mathbb{R}^n$,
\begin{multline*}
 p_l=\frac{1}{2} [\log ( (\sum_i A_{il}e^{p_i}) (\sum_j e^{-p_j}) + \frac{\sum_{i,j} A_{ij} e^{p_i-p_j}}{2\alpha -1}( 1-\alpha 
+e^{ p_l}\frac{\partial \theta}{\partial p_l}(p)) ) \\
		- \log(\sum_j e^{-p_j}) -\log ( (\sum_j A_{lj}e^{-p_j}) (\sum_i e^{p_i}) + \frac{1-\alpha}{2\alpha -1}\sum_{i,j} A_{ij} e^{p_i-p_j} )
		+ \log(\sum_i e^{p_i})] .
\end{multline*}

Let $u$ be the function from $\mathbb{R}^n$ to $\mathbb{R}^n$ defined by
\begin{multline*}
u_l(p)=\frac{1}{2} [\log ( (\sum_i A_{il}e^{p_i}) (\sum_j e^{-p_j}) + \frac{1-\alpha}{2\alpha -1}\sum_{i,j} A_{ij} e^{p_i-p_j} )
		- \log(\sum_j e^{-p_j})\\ -\log ( (\sum_j A_{lj}e^{-p_j}) (\sum_i e^{p_i}) + \frac{1-\alpha}{2\alpha -1}\sum_{i,j} A_{ij} e^{p_i-p_j} )
		+ \log(\sum_i e^{p_i})] \enspace .
\end{multline*}
Using the formula $\log(A)=\log(A+B)-\log(1+B/A)$, we may also write it as
\begin{align*}
 u_l(p)=p_l-\frac{1}{2} \log(1+ d_l \frac{\partial \theta}{\partial p_l}(p)) \enspace .
\end{align*}
where 
\begin{equation} \label{eqn:def-d}
d_l=\frac{e^{ p_l}(\sum_j e^{-p_j})(\sum_{i,j} A_{ij} e^{p_i-p_j})}{(2\alpha -1)(\sum_i A_{il}e^{p_i}) (\sum_j e^{-p_j}) + (1-\alpha)\sum_{i,j} A_{ij} e^{p_i-p_j}} >0\enspace . 
\end{equation}

We can see that the equation $u(p)=p$ is equivalent to
 $\frac{\partial \theta}{\partial p_l}(p)=0$
but also that successively applying the function $u$ corresponds to a descent algorithm.
This is the dual form of Tomlin's algorithm for the computation of HOTS values. 
Note that we do not compute the values of $a_{n+1}$, $b_{n+1}$ and $\mu$
 but give an explicit formula of the fix point operator.

The next proposition gives information on the spectrum of the hessian of the function $\theta$.
\begin{prop} \label{prop:rankD2theta}
 The hessian of the function $\theta$~\eqref{eqn:theta} is symmetric semi-definite with spectral norm smaller that $4$. Its nullspace has dimension~1 exactly for all~$p$ 
and a basis of this nullspace is given by the vector $e$, with all entries equal to~1.
\end{prop}
\begin{proof}
As $\theta$ is convex, its hessian matrix is clearly symmetric semi-definite.

Now, let $\phi: x\mapsto \log (\sum_i e^{x_i})$ be the log-sum-exp function. 

We have $y^T \nabla^2 \phi(x) y =\frac{\sum_i y_i^2 e^{x_i}}{\sum_k e^{x_k}} - \frac{(\sum_i y_i e^{x_i})^2}{(\sum_k e^{x_k})^2}$. This expression is strictly positive for any non constant $y$,
because a constant $y$ is the only equality case of the Cauchy Schwartz inequality $\sum_i y_i e^{x_i/2} e^{x_i/2} \leq (\sum_i e^{x_k})^{1/2} (\sum_i y_i^2 e^{x_i})^{1/2}$.
As the function $\theta$ is the sum of $\phi$ (the third term of~\eqref{eqn:theta}) and of convex functions, it inherits the 
 strict convexity property on spaces without constants.
This development even shows that the kernel of $\nabla^2 \theta(p)$ is of dimension at most~1 for all~$p$.
Finally, as $\theta$ is invariant by addition of a constant, the vector $e$ is clearly part of the nullspace of its hessian. 

For the norm of the hessian matrix, we introduce the linear function $Z: \mathbb{R}^n \to \mathbb{R}^{n \times n}$ such that $(Zp)_{i,j}=p_i -p_j$
and $\tilde{\phi}: z \mapsto \log (\sum_{k\in [n]\times[n]} A_k e^{z_k})$.
Then $\theta(p)=C(\alpha)+(1-\alpha)\phi(p)+(1-\alpha)\phi(-p)+(2\alpha -1)\tilde{\phi}(Zp)$.

$0\leq y^T \nabla^2 \phi (x) y = \frac{\sum_i y_i^2 e^{x_i}}{\sum_k e^{x_k}} - \frac{(\sum_i y_i
 e^{x_i})^2}{(\sum_k e^{x_k})^2} \leq \frac{\sum_i y_i^2 e^{x_i}}{\sum_k e^{x_k}} \leq \norm{y}_{\infty}^2 \leq \norm{y}_{2}^2$.
Thus $\norm{\nabla^2 \phi}_2 \leq 1$. By similar calculations, one gets $y^T Z^T \nabla^2 \tilde{\phi}(x) Z y \leq \norm{Z y}_{\infty}^2$.
As $\norm{Zy}_{\infty}=\max_{i,j} \abs{y_i - y_j} \leq 2 \norm{y}_{\infty}\leq 2 \norm{y}_{2}$,
we have that $\norm{\nabla^2 \theta}_2\leq (1-\alpha) +(1-\alpha) +(2\alpha -1) \times 4 = 6 \alpha -2 < 4$.
Finally, for symmetric matrices, the spectral norm and the operator 2-norm are equal.
\end{proof}

\subsection{Optimization of a scalar function of the HOTS vector} \label{sec:HOTSder}

As for HITS authority in Section~\ref{sec:HITS},
we now consider sets of obligatory links, prohibited links and facultative links.
From now on, the adjacency matrix $A$ may change, so we define $\theta$ and $u$ as functions of $p$ and $A$.
For all $A$, the HOTS vector is uniquely defined up to an additive constant for $\alpha<1$, so we shall set a normalization, like for instance $\sum_i p_i =0$ or 
$\log(\sum_i \exp(p_i))=0$. Thus, given a normalization function $N$, we can define the function 
$p : A \mapsto p(A)$. For all $A$, $i$, $j$, the normalization function $N$ may verify $\frac{\partial N}{\partial p}(p(A)) \frac{\partial p}{\partial A_{i,j}}(A)=0$ 
and $N(p+\lambda)=N(p)+\lambda$ for all $\lambda \in \mathbb{R}$, so that $\frac{\partial N}{\partial p}(p) e=1$.

The HOTS authority optimization problem is:
\begin{equation} \label{eqn:HOTSoptim}
  \max_{J \subseteq \mathcal{F},p \in \mathbb{R}^n,\lambda \in \mathbb{R}} \{ f(p) \; \; ; \;  u(A(J),p) = p \;, \; N(p)=0 \;, 
 \} 
\end{equation} 
We shall mainly study instead the relaxed HOTS authority optimization problem which can be written as:
\begin{equation} \label{eqn:HOTSoptimRelaxed}
\begin{split}
  \max_{A\in \mathbb{R}^{n \times n},p \in \mathbb{R}^n} & f(u) \\
u(A,p) =p \;&, \;\ N(p)=0 \\
 A_{i,j}=1 \; &,\; \forall (i,j) \in \mathcal{O} \\
 A_{i,j}=0\; &,\; \forall (i,j) \in \mathcal{I} \\
0\leq A_{i,j} \leq 1\;&,\;  \forall (i,j) \in \mathcal{F}
\end{split}
\end{equation} 

where %$\mathcal{A}$ is the set of admissible weighted adjacency matrices,
$f:\mathbb{R}^n \to \mathbb{R}$ is the objective function. %and $p(A)$ is the normalized HOTS vector associated to the adjacency matrix $A$.
We will assume that $f$ is differentiable with gradient $\nabla f$.

It is easy to see that $u(A,\cdot)$ is additively homogeneous of degree 1, so the solution $p$ of the equation $u(A,p)=p$ may be seen
as a nonlinear additive eigenvector of $u(A,\cdot)$. In this section, we give the derivative of the HOTS vector with respect to the adjacency matrix.

\begin{prop} \label{prop:hotsDer}
The derivative of $f \circ p$ is given by $g_{i,j}=\sum_l w_l c_{i,j}^l$ where
\[
 w=(-\nabla f^T+ (\nabla f^T e) \nabla N^T) (\frac{\partial^2 \theta}{\partial p^2})^{\#}
\]
and $c_{i,j}^l = \frac{\partial^2 \theta}{\partial p_l \partial A_{i,j}}$.
Moreover, the matrix $(g_{i,j})_{i,j}$ has rank at most 3.
\end{prop}
\begin{proof}
 Let us differentiate 
 with respect to $A_{i,j}$ the equation
\begin{align*}
 p_l(A) &= u_l(A,p(A))=p_l(A)-\frac{1}{2} \log(1+ d_l(A,p(A)) \frac{\partial \theta}{\partial p_l}(A,p(A))) \enspace :\\
 \frac{\partial p_l}{\partial A_{i,j}} &=\frac{\D u_l(A,p(A))}{\D A_{i,j}} \\
&= \frac{\partial p_l}{\partial A_{i,j}} - \frac{1}{2}
\frac{1}{1+ d_l \frac{\partial \theta}{\partial p_l}} 
(\frac{\D d_l}{\D A_{i,j}} \frac{\partial \theta}{\partial p_l}
+d_l \sum_k \frac{\partial^2 \theta}{\partial p_l \partial A_{i,j}}
 +d_l\frac{\partial^2 \theta}{\partial p_l \partial p_k} \frac{\partial p_k}{ \partial A_{i,j}})
\end{align*}
But as $\frac{\partial \theta}{\partial p_l}(A,p(A))=0$ and $d_l(A,p(A))>0$, we get for all $A$, $i$, $j$, $l$:
\[
 \sum_k \frac{\partial^2 \theta}{\partial p_l \partial p_k}(A,p(A)) \frac{\partial p_k}{ \partial A_{i,j}}(A) = -\frac{\partial^2 \theta}{\partial p_l \partial A_{i,j}}(A,p(A))
\]
Or more simply, $ \frac{\partial^2 \theta}{\partial p^2} \frac{\partial p}{ \partial A_{i,j}} = -c_{i,j}$.
Multiplying by $(\frac{\partial^2 \theta}{\partial p^2})^{\#}$, using the fact that the nullspace of $\frac{\partial^2 \theta}{\partial p^2}$ is the multiples of
$e$ (Proposition~\ref{prop:rankD2theta}) and using~\eqref{eqn:S} yields:
\begin{equation} \label{eqn:drazinApplied}
 (I - \frac{e e^T}{n}) \frac{\partial p}{ \partial A_{i,j}} = -(\frac{\partial^2 \theta}{\partial p^2})^{\#}c_{i,j} \enspace .
\end{equation}

Multiplying by $\nabla N^T$ gives $\frac{\nabla N^T e}{n} e^T \frac{\partial p}{ \partial A_{i,j}}= \nabla N^T (\frac{\partial^2 \theta}{\partial p^2})^{\#}c_{i,j} $. We then use $\nabla N^T e=1$, we reinject in~\eqref{eqn:drazinApplied}
 and we multiply by $\nabla f^T$ to get the result $\nabla f^T \frac{\partial p}{ \partial A_{i,j}} = (-\nabla f^T+ (\nabla f^T e) \nabla N^T) (\frac{\partial^2 \theta}{\partial p^2})^{\#}c_{i,j}$.

Finally, the equality
\begin{equation*}
\begin{split}
c_{i,j}^l= \frac{\partial^2 \theta}{\partial p_l \partial A_{i,j}}  = \frac{2\alpha -1}{\sum_{i',j'} A_{i',j'} e^{p_{i'}-p_{j'}}}( -e^{p_i}\delta_{lj}e^{-p_j} + \delta_{li} e^{p_i}e^{-p_j} + B_l e^{p_i}e^{-p_j}) 
\end{split}
\end{equation*}
where
$B_l=\frac{\sum_{i'} A_{i',l} e^{p_{i'}-p_l}-\sum_{j'} A_{l,j'}e^{p_l-p_{j'}}}{\sum_{i',j'} A_{i',j'} e^{p_{i'}-p_{j'}}}$,
 shows that the matrix $(g_{i,j})_{i,j}$ has rank at most~3 since we can write it as the sum of three rank one matrices.
\end{proof}

This proposition is the analog of Corollary~\ref{prop:gcool}, the latter being for Perron vector optimization problems.
It both cases, the derivative has a low rank and one can compute it thanks to a Drazin inverse. 
Moreover, thanks to Proposition~\ref{prop:rankD2theta}, one can apply
Proposition~\ref{prop:wFast} to $M=I_n -\frac{1}{2} \nabla^2\theta$
and $z'=-\frac{1}{2}z$.
Indeed, $M$ has all its eigenvalues within $(-1,1]$, $1$ is a single eigenvalue
with $e e^T/n$ being the associated eigenprojector.
So, for HOTS optimization problems as well as for Perron vector optimization problems, the derivative is easy to compute as soon as the second eigenvalue of $\nabla^2\theta$
is not too small.

Remark that Proposition~\ref{prop:hotsDer} is still true if we replace $\nabla^2\theta$ by $\diag(d) \nabla^2\theta$
and $c_{i,j}$ by $\diag(d) c_{i,j}$. We conjecture that $1$ is still the principal eigenvalue of $I_n -\frac{1}{2} \diag(d) \nabla^2\theta$ and that the spectral gap is larger than the one of $I_n -\frac{1}{2} \nabla^2\theta$. Numerical experiments seem to confirm this conjecture.

% The following proposition gives a particular feature of stationary outlink strategies.
For HOTS optimization also, we have a threshold property.
\begin{prop}[Threshold property] \label{prop:HotsOptStrats}
 Let $A$ be a stationary point for the relaxed HOTS optimization problem~\eqref{eqn:HOTSoptimRelaxed} with associated HOTS vector $p$ and
let $w$ be defined as in
Proposition~\ref{prop:hotsDer}. Let $B=\frac{\sum_{k,l} A_{k,l} e^{p_k - p_l}w_l - \sum_{k,l} w_k A_{k,l} e^{p_k - p_l}}{\sum_{k,l}A_{k,l}e^{p_k - p_l}}$.
Then for all facultative link $(i,j)$, $w_j > w_i + B$ implies that $A_{i,j}=1$
and $w_j < w_i + B$ implies that $A_{i,j}=0$.
\end{prop}
\begin{proof}
A simple development of $c_{i,j}^l$ in Proposition~\ref{prop:hotsDer}, 
shows that the derivative of the objective is
given by $g_{i,j}=\frac{2\alpha-1}{\sum_{k,l}A_{k,l}e^{p_k - p_l}}e^{p_i-p_j}(w_i-w_j+B)$.
 
The result follows from the fact that the problem has only box constraints.
Indeed, a nonzero value of the derivative at the maximum
determines whether the upper bound is saturated ($g_{i,j}<0$) or the lower bound is
saturated ($g_{i,j}>0$). If the derivative is zero, then the weight of the link can take
any value.
\end{proof}

\begin{remark}
 This proposition shows that $w$ gives a total order of preference in pointing to a page or another. 
\end{remark}

\subsection{Example}

We take the same web site as in Figures~\ref{fig:GraphOptHits1} and~\ref{fig:GraphOptHits2} with the same admissible actions.
We choose the objective function $f(p)=\sum_{i\in I} \exp(p_i)$
and the normalization $N(u)=\log(\sum_{i\in I} \exp(p_i))=0$.
The initial value of the objective is 0.142
and we present a local solution with value 0.169 on Figure~\ref{fig:GraphOptHots}.

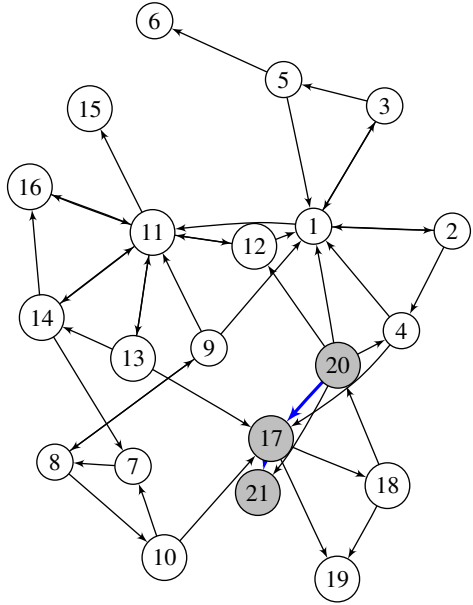
\begin{figure} 
\centering
\small{
\begin{tikzpicture}[>=latex', scale=0.26]
\pgfsetlinewidth{0.5bp}
\input{GraphOptHots.tex} 
\end{tikzpicture}
}
\caption{Strict local maximum for HOTS optimization on 
the~small web graph of Figure~\ref{fig:GraphOptHits1}. 
The locally optimal solution for the problem of maximizing $f(u)=\sum_{i\in I} \exp(p_i)$ presented here is to select the bold arcs with weight~$1$.
If one replaces the arc from 20 to 17 by the arc from 17 to 20
one gets another strict local optimal solution but with a smaller value (0.166 instead of 0.169).
This shows that the problem is not quasi-concave.}\label{fig:GraphOptHots}
\end{figure}

\section{Numerical results} \label{sec:num}

By performing a~crawl of our laboratory website and its surrounding pages with 1,500~pages,
we obtained a~fragment of the web graph. We have
selected 49~pages representing a website~$I$.
We set $r_i=1$ if $i \in I$ and $r_i=0$ otherwise.
The set of obligatory links were the initial links
already present at time of the crawl, the facultative links
are all other links from controlled pages except self-links.
% We assumed that pages without any hyperlink
% should remain so but that we could add any outlink
% on the other pages in $I$.

We launched our numerical experiments on a~personal computer with Intel Xeon CPU at 2.98~Ghz
and 8~GB RAM. We wrote the code in Matlab language.
We refer to~\cite{Fercoq-PRopt} for numerical experiments for PageRank optimization.
% As the algorithm for PageRank optimization is very efficient,
% it can be launched on even larger datasets.

\subsection{HITS authority optimization}

As in Section~\ref{sec:HITS}, we maximize the sum of HITS authority scores on the web site,
that is we maximize $f(u)=\sum_{i \in I} r_i u_i^2$ under the normalization
$N(u)=(\sum_{i \in [n]} r_i u_i^2)^{1/2}=1$.

We use the coupled power and grandient iterations described in Section~\ref{sec:approxGradient}.
We show the progress of the objective on Figure~\ref{fig:hits1500} and we compare 
coupled power and grandient iterations with classical gradient in 
Table \ref{tab:hits1500}.
%From a subgraph of the web graph with 1500 pages, 
%we selected 49 pages, of which we control
%every outlink. We set $R_i=-1$ if $i$ is a controlled page, $R_i=0$ otherwise.

The best strategy of links found has lots of binary values:
only 4569 values different from~1 among 11,516 nonnegative controled values.
Moreover, the heuristic described in Section~\ref{sec:HITS}
gives a 0-1 matrix with a value at 0.07\% from the weighted local optimum found.
It consists in adding all possible links between controlled pages (internal-links)
and some external links. Following Proposition~\ref{prop:shapeHits},
as the controlled pages share many facultative outlinks,
we can identify growing sequences in the sets of activated outlinks.

% Goal : reach 0.21.
% Approx + Armijo : 51 matrix buildings, 1455*2=2910 times the matrix-vector products, 5.5 seconds
% Approx + fix step : 21060 matrix iterations, 21063 times the matrix-vector products, 365 seconds
% Classical Armijo : 191 matrix buildings, 44271  times the matrix-vector products, 76 seconds (7894 times the matrix-vector products for 51 matrix buildings and the result is not better than with the approximations)
% Classical fix step : 21119 matrix buildings, 2574257 times the matrix-vector products, 4736 seconds

\begin{figure} 
\centering
\includegraphics[width=30em]{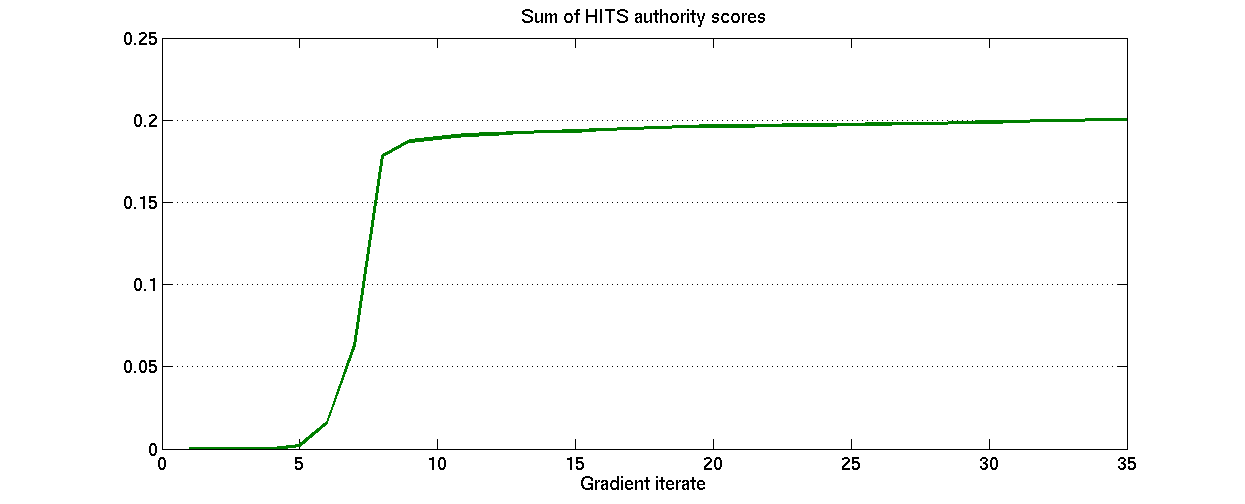}
\caption{Optimization of the sum of HITS authority scores. 
The figure shows that the objective is
increasing during the course of the algorithm. The sum of authority values jumps from 3.5e-6 to 
0.19. However, despite this big progress at the beginning,
 convergence is slow. This is a typical situation with first order methods and nonconvex optimization. 
%The bottom left figure shows the chosen step length in blue and $\frac{1-\sigma}{L_k}$ in green. We can
%see that the step lengths given by Armijo rule are more than 100 times bigger than
%the ones given by fix step length. Finally, 
}\label{fig:hits1500}
\end{figure}

\begin{table}
\centering
\begin{tabular}{|l|c|c|c|}
\hline 
	    & Matrix assemblings & Power iterations & Time \\
\hline
Gradient (Equation \eqref{eqn:wSys}) & 545 & - & 304 s \\
\hline
%Classical, fix step & 21,119 & 2,574,257 & 4,736 s \\
%\hline
%Classical, line search & 245 & 40,070 & 62 s\\
Gradient (Remark~\ref{rem:grad}) & 324 & 56,239 & 67 s\\
\hline
%Approx., fix step & 17,734 & 17,737 & 224 s \\
%\hline
%Approx., line search & 180 & 8,414 & 14 s \\
%Coupled iterations (Algorithm~\ref{alg:armijoBounds}) & 49 & 2,334 & 3.2 s \\
Coupled iterations (Th.~\ref{thm:approxGrad}) & 589 & 14,289 & 15 s \\
\hline
\end{tabular}
 \caption{Comparison of gradient algorithm with the evaluation of the gradient done by 
direct resolution of Equation \eqref{eqn:wSys} by Matlab ``mrdivide'' function,
gradient algorithm with hot started power iterations described in Remark~\ref{rem:grad}
 (precision $10^{-9}$) and coupled gradient and power iterations. 
The goal was to reach the value 0.22 on our laboratory dataset (the best value we found was 0.2285). 
For this problem, coupling the power and gradient iterations
makes a speedup of more than four.} \label{tab:hits1500}
\end{table}

% \begin{figure}
% \hspace{-6em}
% \includegraphics[width=45em]{optimalLinkStrategy.png}
% \caption{Best strategy of links found. Blue dots represent intra-site links, red dots represent inter-site links.
% This strategy has lots of binary values (7337 values different from 1 among 16,341 nonnegative controled values). Moreover, the heuristic described in Section~\ref{sec:HITS}
% gives a 0-1 matrix with a value at 0.03\% from the weighted local optimum found.
% The strategy is very similar from a controlled page to another.}
% \end{figure}

\subsection{HOTS optimization}

Here, we cannot use the coupled power and gradient iterations
since we do not have any effective bound for the distance
between the actual iterate and the true HOTS vector.
We thus solve the HOTS optimization problem with classical gradient.
We computed the derivative by assuming the conjecture at
the end of Section~\ref{sec:HOTSder}. Indeed, the second
eigenvalue of $\nabla^2 \theta$ is small while
the second eigenvalue of $\diag(d) \nabla^2 \theta$ is larger.
We also never encountered a case where the largest eigenvalue of
$\diag(d) \nabla^2 \theta$ is bigger than 4.

We consider the same website as for the numerical 
experiments for HITS authority.
We take as objective function $f(p)=\sum_{i\in I} \exp(p_i)$ under the normalization
$N(u)=\log(\sum_{i \in [n]} \exp(p_i))=0$.

Here, a stationnary point is reached after 55 gradient iterations
and 8.2 s.
The initial value of the objective was 0.0198
and the final value returned by the optimization algorithm was 0.0567.
We give a local optimal adjacency matrix on Figure~\ref{fig:optimHOTS1500}.

\begin{figure}
 \centering
\includegraphics[width=17em]{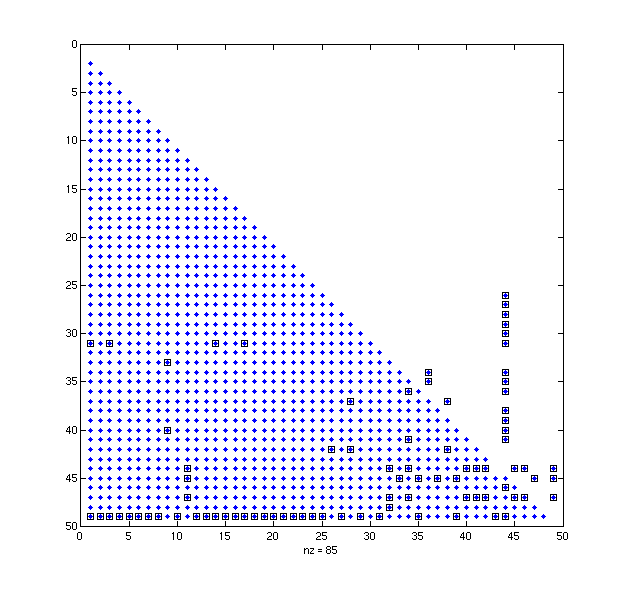}
 \caption{Strict local optimal solution of the HOTS optimization problem. 
We present the adjacency matrix restricted to the set of controlled pages. 
There are no facultative external links added. We sorted pages by their $w$ value. 
The rounded black dots are the obligatory links and the blue dots are the 
activated facultative links. We can see that the higher the $w$ value, 
the bigger the number of controlled pages pointing to that page and 
the smaller the number of pages that page points to (Proposition~\ref{prop:HotsOptStrats}).
It is worth noting that this local solution of the relaxed problem has only binary
weights.
%  There is only one external link and it also the only link with a non binary weight.
%  As the Hessian restricted to the non saturated coordinate is negative,
%  the local optimum is strict. This shows that the problem is not quasi-convex (see Proposition~\ref{prop:HitsNotQuasiConvexConcave}). If we remove this link, we
%  get a binary matrix with nearly the same value.
 }
\label{fig:optimHOTS1500}
\end{figure}

\subsection{Scalability of the algorithms}
%  \begin{table}
% \centering
% \begin{tabular}{|l|c|c|c|c|c|}
% \hline 
% 		&Pages & obligatory links & facultative links & HITS & HOTS \\
% \hline
% CMAP & 1500 &  17,641 & 73,181 & 0.07 s/it  & 0.15 s/it\\
% \hline
% New Zealand Universities & 413,639 & 2,628,894 & 3,048,798 & 10 s/it  & 57 s/it \\
% \hline
% \end{tabular}

\begin{table}[tbph]
%\centering
%\hspace{-4em}
\begin{tabular}{|l|c|c|c|c|} 
\hline 
	 			 	           &  HITS (CMAP) &  HOTS (CMAP) & HITS (NZU) & HOTS (NZU) \\
\hline
Gradient (Eq. \eqref{eqn:wSys})		      &  1.01 s/it & 0.82 s/it     & out of memory & out of memory \\
\hline
Gradient (Rem.~\ref{rem:grad})    		      &  0.23 s/it & 0.12 s/it     & 20 s/it & 62 s/it \\
\hline
Coupled iter. (Th.~\ref{thm:approxGrad})        &  0.038 s/it & 0.04 s/it   & 1.0 s/it & 2.9 s/it \\
\hline
\end{tabular}
 \caption{Mean execution times by gradient iterations for HITS authority and HOTS optimization for 2 fragments of the web graph:
our laboratory web site and surrounding web pages (CMAP dataset, 1,500 pages) and New Zealand Universities websites (NZU dataset, 413,639 pages).
The execution time with a direct resolution of Equation~\eqref{eqn:wSys} by Matlab ``mrdivide''
 function becomes very memory consuming when the size of the problem grows.
It is still acceptable with matrices of size 1,500 but fails for larger matrices.
The other two algorithms scale well. For, HOTS optimization problems,
as they are not Perron vector optimization problems, Theorem~\ref{thm:approxGrad} does not apply.
However, we have implemented the algorithm anyway.
The execution time mainly depends 
on the spectral gap of the matrix, which is similar in the four cases and on the cost by
matrix-vector product, which is growing only linearly thanks to the sparsity of the matrix. 
The number of gradient iterations is in general not dependent on the size of the problem: 
in our experiments for HITS authority optimization, there were even less gradient iterations with the larger dataset.
When the coupled iterations is available, it gives a speedup between 3 and 20.
} \label{tab:hitsBench}
\end{table}

We launched our algorithms on two test sets and we give the execution times on Table~\ref{tab:hitsBench}.
Our laboratory's website is described at the beginning of this section.
The crawl of New Zealand Universities websites is available at~\cite{NZ2006}. 
We selected the 1,696 pages containing the keyword ``maori'' in their url
and 1,807 possible destination pages for the facultative hyperlinks,
which yields 3,048,798 facultative hyperlinks.
In both cases, we maximize the sum of the scores of the controlled pages.

We remark that for a problem more that 300 times bigger,
the computational cost does not increase that much.
Indeed, the spectral gap is similar and the cost by
matrix-vector product is growing only linearly thanks to the sparsity of the matrix.

\bibliographystyle{alpha}
\bibliography{perronOpt,pagerank}

\end{document}

%% file: GraphOptHits1.tex
\pgfsetcolor{black}
  % Edge: 21 -> 18
  \pgfsetcolor{red}
  \draw [->,very thick] (441bp,156bp) .. controls (472bp,159bp) and (518bp,163bp)  .. (562bp,168bp);
  % Edge: 21 -> 20
  \pgfsetcolor{blue}
  \draw [->,very thick] (429bp,178bp) .. controls (435bp,187bp) and (442bp,197bp)  .. (448bp,207bp) .. controls (467bp,238bp) and (484bp,276bp)  .. (499bp,314bp);
  % Edge: 13 -> 17
  \pgfsetcolor{black}
  \draw [->] (254bp,330bp) .. controls (289bp,310bp) and (348bp,275bp)  .. (396bp,247bp);
  % Edge: 20 -> 21
  \draw [->] (499bp,314bp) .. controls (488bp,285bp) and (469bp,242bp)  .. (448bp,207bp) .. controls (444bp,200bp) and (439bp,193bp)  .. (429bp,178bp);
  % Edge: 18 -> 19
  \draw [->] (580bp,141bp) .. controls (570bp,121bp) and (558bp,94bp)  .. (543bp,63bp);
  % Edge: 21 -> 9
  \pgfsetcolor{red}
  \draw [->,very thick] (397bp,184bp) .. controls (383bp,223bp) and (359bp,287bp)  .. (340bp,336bp);
  % Edge: 20 -> 1
  \pgfsetcolor{black}
  \draw [->] (505bp,376bp) .. controls (498bp,410bp) and (488bp,464bp)  .. (479bp,509bp);
  % Edge: 21 -> 17
  \pgfsetcolor{blue}
  \draw [->,very thick] (415bp,185bp) .. controls (415bp,186bp) and (416bp,188bp)  .. (418bp,199bp);
  % Edge: 11 -> 14
  \pgfsetcolor{black}
  \draw [->] (213bp,504bp) .. controls (188bp,479bp) and (147bp,438bp)  .. (112bp,403bp);
  % Edge: 8 -> 10
  \draw [->] (140bp,168bp) .. controls (168bp,145bp) and (217bp,104bp)  .. (257bp,70bp);
  % Edge: 4 -> 1
  \draw [->] (588bp,408bp) .. controls (565bp,434bp) and (524bp,479bp)  .. (492bp,516bp);
  % Edge: 1 -> 2
  \draw [->] (500bp,534bp) .. controls (536bp,533bp) and (599bp,532bp)  .. (649bp,531bp);
  % Edge: 9 -> 11
  \draw [->] (318bp,384bp) .. controls (302bp,411bp) and (276bp,457bp)  .. (252bp,499bp);
  % Edge: 17 -> 9
  \pgfsetcolor{red}
  \draw [->,very thick] (405bp,257bp) .. controls (390bp,278bp) and (368bp,309bp)  .. (346bp,340bp);
  % Edge: 13 -> 14
  \pgfsetcolor{black}
  \draw [->] (195bp,354bp) .. controls (175bp,359bp) and (151bp,365bp)  .. (120bp,372bp);
  % Edge: 21 -> 19
  \pgfsetcolor{red}
  \draw [->,very thick] (432bp,130bp) .. controls (451bp,111bp) and (478bp,84bp)  .. (506bp,56bp);
  % Edge: 16 -> 11
  \pgfsetcolor{black}
  \draw [->] (65bp,538bp) .. controls (100bp,535bp) and (155bp,532bp)  .. (204bp,529bp);
  % Edge: 17 -> 19
  \draw [->] (439bp,202bp) .. controls (457bp,167bp) and (489bp,109bp)  .. (514bp,62bp);
  % Edge: 8 -> 9
  \draw [->] (140bp,202bp) .. controls (178bp,234bp) and (260bp,302bp)  .. (311bp,344bp);
  % Edge: 9 -> 1
  \draw [->] (348bp,381bp) .. controls (373bp,412bp) and (422bp,471bp)  .. (457bp,515bp);
  % Edge: 9 -> 8
  \draw [->] (311bp,344bp) .. controls (273bp,312bp) and (191bp,244bp)  .. (140bp,202bp);
  % Edge: 13 -> 11
  \draw [->] (228bp,379bp) .. controls (230bp,409bp) and (232bp,452bp)  .. (234bp,494bp);
  % Edge: 17 -> 7
  \pgfsetcolor{red}
  \draw [->,dotted,very thick] (393bp,222bp) .. controls (358bp,213bp) and (302bp,196bp)  .. (255bp,183bp);
  \pgfsetcolor{black}
  \draw (321bp,216bp) node {0.18};
  % Edge: 12 -> 1
  \draw [->] (412bp,525bp) .. controls (421bp,526bp) and (430bp,528bp)  .. (448bp,531bp);
  % Edge: 11 -> 13
  \draw [->] (234bp,494bp) .. controls (232bp,464bp) and (230bp,421bp)  .. (228bp,379bp);
  % Edge: 2 -> 4
  \draw [->] (664bp,507bp) .. controls (652bp,484bp) and (634bp,447bp)  .. (617bp,412bp);
  % Edge: 17 -> 20
  \pgfsetcolor{blue}
  \draw [->,very thick] (444bp,257bp) .. controls (456bp,272bp) and (472bp,293bp)  .. (491bp,318bp);
  % Edge: 17 -> 18
  \pgfsetcolor{black}
  \draw [->] (454bp,220bp) .. controls (482bp,210bp) and (523bp,196bp)  .. (563bp,182bp);
  % Edge: 20 -> 12
  \draw [->] (492bp,370bp) .. controls (469bp,400bp) and (431bp,452bp)  .. (399bp,494bp);
  % Edge: 12 -> 11
  \draw [->] (348bp,522bp) .. controls (327bp,523bp) and (301bp,524bp)  .. (268bp,526bp);
  % Edge: 14 -> 16
  \draw [->] (78bp,411bp) .. controls (69bp,436bp) and (57bp,472bp)  .. (44bp,509bp);
  % Edge: 7 -> 8
  \draw [->] (204bp,178bp) .. controls (189bp,180bp) and (172bp,181bp)  .. (146bp,183bp);
  % Edge: 11 -> 16
  \draw [->] (204bp,529bp) .. controls (169bp,532bp) and (114bp,535bp)  .. (65bp,538bp);
  % Edge: 14 -> 7
  \draw [->] (107bp,353bp) .. controls (134bp,315bp) and (182bp,246bp)  .. (215bp,198bp);
  % Edge: 1 -> 11
  \draw [->] (449bp,543bp) .. controls (430bp,548bp) and (404bp,554bp)  .. (380bp,554bp) .. controls (345bp,555bp) and (305bp,547bp)  .. (267bp,536bp);
  % Edge: 10 -> 17
  \draw [->] (302bp,75bp) .. controls (327bp,107bp) and (370bp,162bp)  .. (404bp,205bp);
  % Edge: 18 -> 20
  \draw [->] (580bp,200bp) .. controls (566bp,229bp) and (544bp,274bp)  .. (525bp,315bp);
  % Edge: 4 -> 17
  \draw [->] (587bp,369bp) .. controls (573bp,356bp) and (553bp,336bp)  .. (535bp,320bp) .. controls (510bp,298bp) and (481bp,274bp)  .. (450bp,251bp);
  % Edge: 11 -> 12
  \draw [->] (268bp,525bp) .. controls (289bp,524bp) and (315bp,523bp)  .. (348bp,521bp);
  % Edge: 17 -> 21
  \pgfsetcolor{blue}
  \draw [->,very thick] (418bp,199bp) .. controls (418bp,198bp) and (417bp,196bp)  .. (415bp,185bp);
  % Edge: 5 -> 6
  \pgfsetcolor{black}
  \draw [->] (424bp,763bp) .. controls (395bp,781bp) and (344bp,814bp)  .. (302bp,840bp);
  % Edge: 10 -> 7
  \draw [->] (270bp,79bp) .. controls (262bp,98bp) and (252bp,123bp)  .. (240bp,152bp);
  % Edge: 20 -> 17
  \pgfsetcolor{blue}
  \draw [->,very thick] (491bp,318bp) .. controls (479bp,303bp) and (463bp,282bp)  .. (444bp,257bp);
  % Edge: 2 -> 1
  \pgfsetcolor{black}
  \draw [->] (649bp,532bp) .. controls (613bp,533bp) and (550bp,534bp)  .. (500bp,535bp);
  % Edge: 1 -> 3
  \draw [->] (488bp,557bp) .. controls (507bp,587bp) and (543bp,641bp)  .. (571bp,684bp);
  % Edge: 3 -> 1
  \draw [->] (571bp,684bp) .. controls (552bp,654bp) and (516bp,600bp)  .. (488bp,557bp);
  % Edge: 20 -> 4
  \draw [->] (540bp,358bp) .. controls (551bp,363bp) and (562bp,368bp)  .. (582bp,377bp);
  % Edge: 14 -> 11
  \draw [->] (112bp,403bp) .. controls (137bp,428bp) and (178bp,469bp)  .. (213bp,504bp);
  % Edge: 11 -> 15
  \draw [->] (228bp,559bp) .. controls (219bp,594bp) and (204bp,651bp)  .. (191bp,699bp);
  % Edge: 21 -> 7
  \pgfsetcolor{red}
  \draw [->,very thick] (377bp,157bp) .. controls (346bp,162bp) and (299bp,168bp)  .. (256bp,173bp);
  % Edge: 3 -> 5
  \pgfsetcolor{black}
  \draw [->] (560bp,714bp) .. controls (538bp,721bp) and (505bp,730bp)  .. (471bp,741bp);
  % Edge: 5 -> 1
  \draw [->] (449bp,723bp) .. controls (455bp,685bp) and (463bp,614bp)  .. (471bp,561bp);
  % Node: 11
\begin{scope}
  \pgfsetstrokecolor{black}
  \draw (236bp,527bp) ellipse (32bp and 33bp);
  \draw (236bp,527bp) node {11};
\end{scope}
  % Node: 10
\begin{scope}
  \pgfsetstrokecolor{black}
  \draw (282bp,49bp) ellipse (32bp and 33bp);
  \draw (282bp,49bp) node {10};
\end{scope}
  % Node: 13
\begin{scope}
  \pgfsetstrokecolor{black}
  \draw (226bp,346bp) ellipse (32bp and 33bp);
  \draw (226bp,346bp) node {13};
\end{scope}
  % Node: 12
\begin{scope}
  \pgfsetstrokecolor{black}
  \draw (380bp,520bp) ellipse (32bp and 33bp);
  \draw (380bp,520bp) node {12};
\end{scope}
  % Node: 20
\begin{scope}
  \pgfsetstrokecolor{black}
  \pgfsetfillcolor{lightgray}
  \filldraw (511bp,344bp) ellipse (32bp and 33bp);
  \draw (511bp,344bp) node {20};
\end{scope}
  % Node: 21
\begin{scope}
  \pgfsetstrokecolor{black}
  \pgfsetfillcolor{lightgray}
  \filldraw (409bp,153bp) ellipse (32bp and 33bp);
  \draw (409bp,153bp) node {21};
\end{scope}
  % Node: 17
\begin{scope}
  \pgfsetstrokecolor{black}
  \pgfsetfillcolor{lightgray}
  \filldraw (424bp,231bp) ellipse (32bp and 33bp);
  \draw (424bp,231bp) node {17};
\end{scope}
  % Node: 16
\begin{scope}
  \pgfsetstrokecolor{black}
  \draw (33bp,540bp) ellipse (32bp and 33bp);
  \draw (33bp,540bp) node {16};
\end{scope}
  % Node: 19
\begin{scope}
  \pgfsetstrokecolor{black}
  \draw (529bp,33bp) ellipse (32bp and 33bp);
  \draw (529bp,33bp) node {19};
\end{scope}
  % Node: 18
\begin{scope}
  \pgfsetstrokecolor{black}
  \draw (594bp,171bp) ellipse (32bp and 33bp);
  \draw (594bp,171bp) node {18};
\end{scope}
  % Node: 15
\begin{scope}
  \pgfsetstrokecolor{black}
  \draw (183bp,730bp) ellipse (32bp and 33bp);
  \draw (183bp,730bp) node {15};
\end{scope}
  % Node: 1
\begin{scope}
  \pgfsetstrokecolor{black}
  \draw (474bp,535bp) ellipse (26bp and 26bp);
  \draw (474bp,535bp) node {1};
\end{scope}
  % Node: 3
\begin{scope}
  \pgfsetstrokecolor{black}
  \draw (585bp,706bp) ellipse (26bp and 26bp);
  \draw (585bp,706bp) node {3};
\end{scope}
  % Node: 2
\begin{scope}
  \pgfsetstrokecolor{black}
  \draw (675bp,531bp) ellipse (26bp and 26bp);
  \draw (675bp,531bp) node {2};
\end{scope}
  % Node: 5
\begin{scope}
  \pgfsetstrokecolor{black}
  \draw (446bp,749bp) ellipse (26bp and 26bp);
  \draw (446bp,749bp) node {5};
\end{scope}
  % Node: 4
\begin{scope}
  \pgfsetstrokecolor{black}
  \draw (606bp,388bp) ellipse (26bp and 26bp);
  \draw (606bp,388bp) node {4};
\end{scope}
  % Node: 7
\begin{scope}
  \pgfsetstrokecolor{black}
  \draw (230bp,176bp) ellipse (26bp and 26bp);
  \draw (230bp,176bp) node {7};
\end{scope}
  % Node: 6
\begin{scope}
  \pgfsetstrokecolor{black}
  \draw (280bp,854bp) ellipse (26bp and 26bp);
  \draw (280bp,854bp) node {6};
\end{scope}
  % Node: 9
\begin{scope}
  \pgfsetstrokecolor{black}
  \draw (331bp,361bp) ellipse (26bp and 26bp);
  \draw (331bp,361bp) node {9};
\end{scope}
  % Node: 8
\begin{scope}
  \pgfsetstrokecolor{black}
  \draw (120bp,185bp) ellipse (26bp and 26bp);
  \draw (120bp,185bp) node {8};
\end{scope}
  % Node: 14
\begin{scope}
  \pgfsetstrokecolor{black}
  \draw (89bp,380bp) ellipse (32bp and 33bp);
  \draw (89bp,380bp) node {14};
\end{scope}

%% file: GraphOptHits2.tex
\pgfsetcolor{black}
  % Edge: 21 -> 18
  \pgfsetcolor{red}
  \draw [->,very thick] (392bp,160bp) .. controls (423bp,162bp) and (470bp,164bp)  .. (514bp,166bp);
  % Edge: 21 -> 20
  \pgfsetcolor{blue}
  \draw [->,very thick] (381bp,183bp) .. controls (389bp,192bp) and (396bp,202bp)  .. (403bp,212bp) .. controls (423bp,242bp) and (442bp,277bp)  .. (461bp,313bp);
  % Edge: 13 -> 17
  \pgfsetcolor{black}
  \draw [->] (209bp,336bp) .. controls (244bp,316bp) and (304bp,281bp)  .. (351bp,252bp);
  % Edge: 10 -> 17
  \pgfsetcolor{black}
  \draw [->] (248bp,88bp) .. controls (277bp,121bp) and (328bp,179bp)  .. (357bp,212bp);
  % Edge: 17 -> 10
  \pgfsetcolor{red}
  \draw [->,dotted,very thick] (357bp,212bp) .. controls (323bp,174bp) and (282bp,127bp)  .. (248bp,88bp);
  \pgfsetstrokecolor{black}
  \pgfsetcolor{black}
  \draw (288bp,163bp) node {0.23};
  % Edge: 20 -> 21
  \draw [->] (461bp,313bp) .. controls (447bp,286bp) and (425bp,245bp)  .. (403bp,212bp) .. controls (398bp,205bp) and (393bp,198bp)  .. (381bp,183bp);
  % Edge: 18 -> 19
  \draw [->] (531bp,139bp) .. controls (521bp,119bp) and (506bp,93bp)  .. (489bp,62bp);
  % Edge: 21 -> 9
  \pgfsetcolor{red}
  \draw [->,very thick] (350bp,189bp) .. controls (337bp,228bp) and (315bp,293bp)  .. (298bp,342bp);
  % Edge: 20 -> 1
  \pgfsetcolor{black}
  \draw [->] (469bp,374bp) .. controls (463bp,411bp) and (453bp,470bp)  .. (444bp,517bp);
  % Edge: 21 -> 17
  \pgfsetcolor{blue}
  \draw [->,very thick] (368bp,190bp) .. controls (368bp,191bp) and (369bp,193bp)  .. (371bp,204bp);
  % Edge: 11 -> 14
  \pgfsetcolor{black}
  \draw [->] (183bp,514bp) .. controls (157bp,493bp) and (114bp,460bp)  .. (76bp,431bp);
  % Edge: 8 -> 10
  \draw [->] (89bp,185bp) .. controls (116bp,162bp) and (163bp,120bp)  .. (202bp,86bp);
  % Edge: 4 -> 1
  \draw [->] (549bp,412bp) .. controls (527bp,439bp) and (488bp,485bp)  .. (457bp,523bp);
  % Edge: 21 -> 10
  \pgfsetcolor{red}
  \draw [->,very thick] (334bp,139bp) .. controls (313bp,125bp) and (283bp,104bp)  .. (252bp,82bp);
  % Edge: 1 -> 2
  \pgfsetcolor{black}
  \draw [->] (466bp,542bp) .. controls (501bp,541bp) and (563bp,538bp)  .. (613bp,537bp);
  % Edge: 9 -> 11
  \draw [->] (279bp,391bp) .. controls (265bp,418bp) and (243bp,463bp)  .. (223bp,505bp);
  % Edge: 17 -> 9
  \pgfsetcolor{red}
  \draw [->,very thick] (361bp,263bp) .. controls (346bp,284bp) and (326bp,314bp)  .. (305bp,345bp);
  % Edge: 13 -> 14
  \pgfsetcolor{black}
  \draw [->] (152bp,365bp) .. controls (133bp,374bp) and (109bp,385bp)  .. (79bp,398bp);
  % Edge: 21 -> 19
  \pgfsetcolor{red}
  \draw [->,very thick] (382bp,134bp) .. controls (400bp,114bp) and (426bp,86bp)  .. (452bp,57bp);
  % Edge: 16 -> 11
  \pgfsetcolor{black}
  \draw [->] (63bp,589bp) .. controls (93bp,578bp) and (138bp,562bp)  .. (179bp,545bp);
  % Edge: 17 -> 19
  \draw [->] (393bp,206bp) .. controls (410bp,171bp) and (437bp,111bp)  .. (460bp,63bp);
  % Edge: 8 -> 9
  \draw [->] (90bp,218bp) .. controls (130bp,247bp) and (216bp,312bp)  .. (269bp,351bp);
  % Edge: 9 -> 1
  \draw [->] (307bp,387bp) .. controls (334bp,419bp) and (386bp,480bp)  .. (423bp,523bp);
  % Edge: 9 -> 8
  \draw [->] (269bp,351bp) .. controls (229bp,322bp) and (143bp,257bp)  .. (90bp,218bp);
  % Edge: 13 -> 11
  \draw [->] (186bp,384bp) .. controls (191bp,414bp) and (197bp,459bp)  .. (204bp,502bp);
  % Edge: 12 -> 1
  \draw [->] (386bp,524bp) .. controls (393bp,527bp) and (399bp,529bp)  .. (415bp,535bp);
  % Edge: 11 -> 13
  \draw [->] (204bp,502bp) .. controls (199bp,472bp) and (193bp,427bp)  .. (186bp,384bp);
  % Edge: 2 -> 4
  \draw [->] (627bp,513bp) .. controls (615bp,489bp) and (596bp,451bp)  .. (578bp,416bp);
  % Edge: 17 -> 20
  \pgfsetcolor{blue}
  \draw [->,very thick] (401bp,260bp) .. controls (415bp,275bp) and (432bp,294bp)  .. (453bp,318bp);
  % Edge: 17 -> 18
  \pgfsetcolor{black}
  \draw [->] (409bp,224bp) .. controls (436bp,212bp) and (477bp,196bp)  .. (516bp,180bp);
  % Edge: 20 -> 12
  \draw [->] (456bp,369bp) .. controls (436bp,398bp) and (402bp,446bp)  .. (374bp,487bp);
  % Edge: 12 -> 11
  \draw [->] (323bp,518bp) .. controls (302bp,521bp) and (274bp,525bp)  .. (241bp,530bp);
  % Edge: 14 -> 16
  \draw [->] (47bp,443bp) .. controls (44bp,475bp) and (40bp,523bp)  .. (36bp,567bp);
  % Edge: 7 -> 8
  \draw [->] (155bp,197bp) .. controls (140bp,198bp) and (121bp,199bp)  .. (95bp,201bp);
  % Edge: 11 -> 16
  \draw [->] (179bp,545bp) .. controls (149bp,556bp) and (104bp,572bp)  .. (63bp,589bp);
  % Edge: 14 -> 7
  \draw [->] (67bp,383bp) .. controls (91bp,343bp) and (136bp,269bp)  .. (167bp,218bp);
  % Edge: 1 -> 11
  \draw [->] (414bp,546bp) .. controls (397bp,547bp) and (375bp,549bp)  .. (355bp,548bp) .. controls (320bp,547bp) and (280bp,543bp)  .. (241bp,539bp);
  % Edge: 18 -> 20
  \pgfsetcolor{black}
  \draw [->] (534bp,198bp) .. controls (522bp,227bp) and (504bp,271bp)  .. (487bp,312bp);
  % Edge: 4 -> 17
  \draw [->] (550bp,372bp) .. controls (537bp,356bp) and (518bp,334bp)  .. (499bp,318bp) .. controls (473bp,296bp) and (440bp,273bp)  .. (407bp,253bp);
  % Edge: 11 -> 12
  \draw [->] (241bp,530bp) .. controls (262bp,527bp) and (290bp,523bp)  .. (323bp,518bp);
  % Edge: 17 -> 21
  \pgfsetcolor{blue}
  \draw [->,very thick] (371bp,204bp) .. controls (371bp,203bp) and (370bp,201bp)  .. (368bp,190bp);
  % Edge: 5 -> 6
  \pgfsetcolor{black}
  \draw [->] (373bp,766bp) .. controls (340bp,781bp) and (282bp,808bp)  .. (236bp,829bp);
  % Edge: 10 -> 7
  \draw [->] (215bp,95bp) .. controls (208bp,115bp) and (200bp,141bp)  .. (189bp,171bp);
  % Edge: 20 -> 17
  \pgfsetcolor{blue}
  \draw [->,very thick] (453bp,318bp) .. controls (439bp,303bp) and (422bp,284bp)  .. (401bp,260bp);
  % Edge: 2 -> 1
  \pgfsetcolor{black}
  \draw [->] (613bp,537bp) .. controls (578bp,538bp) and (516bp,541bp)  .. (466bp,542bp);
  % Edge: 1 -> 3
  \draw [->] (453bp,566bp) .. controls (471bp,596bp) and (503bp,651bp)  .. (529bp,695bp);
  % Edge: 3 -> 1
  \draw [->] (529bp,694bp) .. controls (511bp,664bp) and (479bp,609bp)  .. (453bp,565bp);
  % Edge: 20 -> 4
  \draw [->] (503bp,358bp) .. controls (513bp,363bp) and (524bp,369bp)  .. (543bp,379bp);
  % Edge: 14 -> 11
  \draw [->] (76bp,431bp) .. controls (102bp,452bp) and (145bp,485bp)  .. (183bp,514bp);
  % Edge: 11 -> 15
  \draw [->] (194bp,563bp) .. controls (179bp,593bp) and (155bp,641bp)  .. (134bp,684bp);
  % Edge: 3 -> 5
  \draw [->] (517bp,724bp) .. controls (493bp,730bp) and (458bp,739bp)  .. (422bp,748bp);
  % Edge: 5 -> 1
  \draw [->] (402bp,729bp) .. controls (410bp,692bp) and (425bp,621bp)  .. (435bp,569bp);
  % Node: 11
\begin{scope}
  \pgfsetstrokecolor{black}
  \draw (209bp,534bp) ellipse (32bp and 33bp);
  \draw (209bp,534bp) node {11};
\end{scope}
  % Node: 10
\begin{scope}
  \pgfsetstrokecolor{black}
  \draw (226bp,64bp) ellipse (32bp and 33bp);
  \draw (226bp,64bp) node {10};
\end{scope}
  % Node: 13
\begin{scope}
  \pgfsetstrokecolor{black}
  \draw (181bp,352bp) ellipse (32bp and 33bp);
  \draw (181bp,352bp) node {13};
\end{scope}
  % Node: 12
\begin{scope}
  \pgfsetstrokecolor{black}
  \draw (355bp,514bp) ellipse (32bp and 33bp);
  \draw (355bp,514bp) node {12};
\end{scope}
  % Node: 20
\begin{scope}
  \pgfsetstrokecolor{black}
  \pgfsetfillcolor{lightgray}
  \filldraw (475bp,342bp) ellipse (32bp and 33bp);
  \draw (475bp,342bp) node {20};
\end{scope}
  % Node: 21
\begin{scope}
  \pgfsetstrokecolor{black}
  \pgfsetfillcolor{lightgray}
  \filldraw (360bp,158bp) ellipse (32bp and 33bp);
  \draw (360bp,158bp) node {21};
\end{scope}
  % Node: 17
\begin{scope}
  \pgfsetstrokecolor{black}
  \pgfsetfillcolor{lightgray}
  \filldraw (379bp,236bp) ellipse (32bp and 33bp);
  \draw (379bp,236bp) node {17};
\end{scope}
  % Node: 16
\begin{scope}
  \pgfsetstrokecolor{black}
  \draw (33bp,600bp) ellipse (32bp and 33bp);
  \draw (33bp,600bp) node {16};
\end{scope}
  % Node: 19
\begin{scope}
  \pgfsetstrokecolor{black}
  \draw (474bp,33bp) ellipse (32bp and 33bp);
  \draw (474bp,33bp) node {19};
\end{scope}
  % Node: 18
\begin{scope}
  \pgfsetstrokecolor{black}
  \draw (546bp,168bp) ellipse (32bp and 33bp);
  \draw (546bp,168bp) node {18};
\end{scope}
  % Node: 15
\begin{scope}
  \pgfsetstrokecolor{black}
  \draw (119bp,713bp) ellipse (32bp and 33bp);
  \draw (119bp,713bp) node {15};
\end{scope}
  % Node: 1
\begin{scope}
  \pgfsetstrokecolor{black}
  \draw (440bp,543bp) ellipse (26bp and 26bp);
  \draw (440bp,543bp) node {1};
\end{scope}
  % Node: 3
\begin{scope}
  \pgfsetstrokecolor{black}
  \draw (542bp,717bp) ellipse (26bp and 26bp);
  \draw (542bp,717bp) node {3};
\end{scope}
  % Node: 2
\begin{scope}
  \pgfsetstrokecolor{black}
  \draw (639bp,536bp) ellipse (26bp and 26bp);
  \draw (639bp,536bp) node {2};
\end{scope}
  % Node: 5
\begin{scope}
  \pgfsetstrokecolor{black}
  \draw (397bp,755bp) ellipse (26bp and 26bp);
  \draw (397bp,755bp) node {5};
\end{scope}
  % Node: 4
\begin{scope}
  \pgfsetstrokecolor{black}
  \draw (566bp,392bp) ellipse (26bp and 26bp);
  \draw (566bp,392bp) node {4};
\end{scope}
  % Node: 7
\begin{scope}
  \pgfsetstrokecolor{black}
  \draw (181bp,196bp) ellipse (26bp and 26bp);
  \draw (181bp,196bp) node {7};
\end{scope}
  % Node: 6
\begin{scope}
  \pgfsetstrokecolor{black}
  \draw (212bp,840bp) ellipse (26bp and 26bp);
  \draw (212bp,840bp) node {6};
\end{scope}
  % Node: 9
\begin{scope}
  \pgfsetstrokecolor{black}
  \draw (290bp,367bp) ellipse (26bp and 26bp);
  \draw (290bp,367bp) node {9};
\end{scope}
  % Node: 8
\begin{scope}
  \pgfsetstrokecolor{black}
  \draw (69bp,202bp) ellipse (26bp and 26bp);
  \draw (69bp,202bp) node {8};
\end{scope}
  % Node: 14
\begin{scope}
  \pgfsetstrokecolor{black}
  \draw (50bp,411bp) ellipse (32bp and 33bp);
  \draw (50bp,411bp) node {14};
\end{scope}

%% file: GraphOptHots.tex
\pgfsetcolor{black}
  % Edge: 13 -> 17
  \draw [->] (209bp,336bp) .. controls (244bp,316bp) and (304bp,281bp)  .. (351bp,252bp);
  % Edge: 20 -> 21
  \draw [->] (461bp,313bp) .. controls (447bp,286bp) and (425bp,245bp)  .. (403bp,212bp) .. controls (398bp,205bp) and (393bp,198bp)  .. (381bp,183bp);
  % Edge: 18 -> 19
  \draw [->] (531bp,139bp) .. controls (521bp,119bp) and (506bp,93bp)  .. (489bp,62bp);
  % Edge: 20 -> 1
  \draw [->] (469bp,374bp) .. controls (463bp,411bp) and (453bp,470bp)  .. (444bp,517bp);
  % Edge: 11 -> 14
  \draw [->] (183bp,514bp) .. controls (157bp,493bp) and (114bp,460bp)  .. (76bp,431bp);
  % Edge: 8 -> 10
  \draw [->] (89bp,185bp) .. controls (116bp,162bp) and (163bp,120bp)  .. (202bp,86bp);
  % Edge: 4 -> 1
  \draw [->] (549bp,412bp) .. controls (527bp,439bp) and (488bp,485bp)  .. (457bp,523bp);
  % Edge: 1 -> 2
  \draw [->] (466bp,542bp) .. controls (501bp,541bp) and (563bp,538bp)  .. (613bp,537bp);
  % Edge: 9 -> 11
  \draw [->] (279bp,391bp) .. controls (265bp,418bp) and (243bp,463bp)  .. (223bp,505bp);
  % Edge: 13 -> 14
  \draw [->] (152bp,365bp) .. controls (133bp,374bp) and (109bp,385bp)  .. (79bp,398bp);
  % Edge: 16 -> 11
  \draw [->] (63bp,589bp) .. controls (93bp,578bp) and (138bp,562bp)  .. (179bp,545bp);
  % Edge: 17 -> 19
  \draw [->] (393bp,206bp) .. controls (410bp,171bp) and (437bp,111bp)  .. (460bp,63bp);
  % Edge: 8 -> 9
  \draw [->] (90bp,218bp) .. controls (130bp,247bp) and (216bp,312bp)  .. (269bp,351bp);
  % Edge: 9 -> 1
  \draw [->] (307bp,387bp) .. controls (334bp,419bp) and (386bp,480bp)  .. (423bp,523bp);
  % Edge: 9 -> 8
  \draw [->] (269bp,351bp) .. controls (229bp,322bp) and (143bp,257bp)  .. (90bp,218bp);
  % Edge: 13 -> 11
  \draw [->] (186bp,384bp) .. controls (191bp,414bp) and (197bp,459bp)  .. (204bp,502bp);
  % Edge: 12 -> 1
  \draw [->] (386bp,524bp) .. controls (393bp,527bp) and (399bp,529bp)  .. (415bp,535bp);
  % Edge: 11 -> 13
  \draw [->] (204bp,502bp) .. controls (199bp,472bp) and (193bp,427bp)  .. (186bp,384bp);
  % Edge: 2 -> 4
  \draw [->] (627bp,513bp) .. controls (615bp,489bp) and (596bp,451bp)  .. (578bp,416bp);
  % Edge: 17 -> 18
  \draw [->] (409bp,224bp) .. controls (436bp,212bp) and (477bp,196bp)  .. (516bp,180bp);
  % Edge: 20 -> 12
  \draw [->] (456bp,369bp) .. controls (436bp,398bp) and (402bp,446bp)  .. (374bp,487bp);
  % Edge: 12 -> 11
  \draw [->] (323bp,518bp) .. controls (302bp,521bp) and (274bp,525bp)  .. (241bp,530bp);
  % Edge: 14 -> 16
  \draw [->] (47bp,443bp) .. controls (44bp,475bp) and (40bp,523bp)  .. (36bp,567bp);
  % Edge: 7 -> 8
  \draw [->] (155bp,197bp) .. controls (140bp,198bp) and (121bp,199bp)  .. (95bp,201bp);
  % Edge: 11 -> 16
  \draw [->] (179bp,545bp) .. controls (149bp,556bp) and (104bp,572bp)  .. (63bp,589bp);
  % Edge: 14 -> 7
  \draw [->] (67bp,383bp) .. controls (91bp,343bp) and (136bp,269bp)  .. (167bp,218bp);
  % Edge: 1 -> 11
  \draw [->] (414bp,546bp) .. controls (397bp,547bp) and (375bp,549bp)  .. (355bp,548bp) .. controls (320bp,547bp) and (280bp,543bp)  .. (241bp,539bp);
  % Edge: 10 -> 17
  \draw [->] (248bp,88bp) .. controls (275bp,119bp) and (321bp,171bp)  .. (357bp,212bp);
  % Edge: 18 -> 20
  \draw [->] (534bp,198bp) .. controls (522bp,227bp) and (504bp,271bp)  .. (487bp,312bp);
  % Edge: 4 -> 17
  \draw [->] (550bp,372bp) .. controls (537bp,356bp) and (518bp,334bp)  .. (499bp,318bp) .. controls (473bp,296bp) and (440bp,273bp)  .. (407bp,253bp);
  % Edge: 11 -> 12
  \draw [->] (241bp,530bp) .. controls (262bp,527bp) and (290bp,523bp)  .. (323bp,518bp);
  % Edge: 17 -> 21
  \pgfsetcolor{blue}
  \draw [->,very thick] (371bp,204bp) .. controls (371bp,203bp) and (370bp,201bp)  .. (368bp,190bp);
  % Edge: 5 -> 6
  \pgfsetcolor{black}
  \draw [->] (373bp,766bp) .. controls (340bp,781bp) and (282bp,808bp)  .. (236bp,829bp);
  % Edge: 10 -> 7
  \draw [->] (215bp,95bp) .. controls (208bp,115bp) and (200bp,141bp)  .. (189bp,171bp);
  % Edge: 20 -> 17
  \pgfsetcolor{blue}
  \draw [->,very thick] (453bp,318bp) .. controls (439bp,303bp) and (422bp,284bp)  .. (401bp,260bp);
  % Edge: 2 -> 1
  \pgfsetcolor{black}
  \draw [->] (613bp,537bp) .. controls (578bp,538bp) and (516bp,541bp)  .. (466bp,542bp);
  % Edge: 1 -> 3
  \draw [->] (453bp,566bp) .. controls (471bp,596bp) and (503bp,651bp)  .. (529bp,695bp);
  % Edge: 3 -> 1
  \draw [->] (529bp,694bp) .. controls (511bp,664bp) and (479bp,609bp)  .. (453bp,565bp);
  % Edge: 20 -> 4
  \draw [->] (503bp,358bp) .. controls (513bp,363bp) and (524bp,369bp)  .. (543bp,379bp);
  % Edge: 14 -> 11
  \draw [->] (76bp,431bp) .. controls (102bp,452bp) and (145bp,485bp)  .. (183bp,514bp);
  % Edge: 11 -> 15
  \draw [->] (194bp,563bp) .. controls (179bp,593bp) and (155bp,641bp)  .. (134bp,684bp);
  % Edge: 3 -> 5
  \draw [->] (517bp,724bp) .. controls (493bp,730bp) and (458bp,739bp)  .. (422bp,748bp);
  % Edge: 5 -> 1
  \draw [->] (402bp,729bp) .. controls (410bp,692bp) and (425bp,621bp)  .. (435bp,569bp);
  % Node: 11
\begin{scope}
  \pgfsetstrokecolor{black}
  \draw (209bp,534bp) ellipse (32bp and 33bp);
  \draw (209bp,534bp) node {11};
\end{scope}
  % Node: 10
\begin{scope}
  \pgfsetstrokecolor{black}
  \draw (226bp,64bp) ellipse (32bp and 33bp);
  \draw (226bp,64bp) node {10};
\end{scope}
  % Node: 13
\begin{scope}
  \pgfsetstrokecolor{black}
  \draw (181bp,352bp) ellipse (32bp and 33bp);
  \draw (181bp,352bp) node {13};
\end{scope}
  % Node: 12
\begin{scope}
  \pgfsetstrokecolor{black}
  \draw (355bp,514bp) ellipse (32bp and 33bp);
  \draw (355bp,514bp) node {12};
\end{scope}
  % Node: 20
\begin{scope}
  \pgfsetstrokecolor{black}
  \pgfsetfillcolor{lightgray}
  \filldraw (475bp,342bp) ellipse (32bp and 33bp);
  \draw (475bp,342bp) node {20};
\end{scope}
  % Node: 21
\begin{scope}
  \pgfsetstrokecolor{black}
  \pgfsetfillcolor{lightgray}
  \filldraw (360bp,158bp) ellipse (32bp and 33bp);
  \draw (360bp,158bp) node {21};
\end{scope}
  % Node: 17
\begin{scope}
  \pgfsetstrokecolor{black}
  \pgfsetfillcolor{lightgray}
  \filldraw (379bp,236bp) ellipse (32bp and 33bp);
  \draw (379bp,236bp) node {17};
\end{scope}
  % Node: 16
\begin{scope}
  \pgfsetstrokecolor{black}
  \draw (33bp,600bp) ellipse (32bp and 33bp);
  \draw (33bp,600bp) node {16};
\end{scope}
  % Node: 19
\begin{scope}
  \pgfsetstrokecolor{black}
  \draw (474bp,33bp) ellipse (32bp and 33bp);
  \draw (474bp,33bp) node {19};
\end{scope}
  % Node: 18
\begin{scope}
  \pgfsetstrokecolor{black}
  \draw (546bp,168bp) ellipse (32bp and 33bp);
  \draw (546bp,168bp) node {18};
\end{scope}
  % Node: 15
\begin{scope}
  \pgfsetstrokecolor{black}
  \draw (119bp,713bp) ellipse (32bp and 33bp);
  \draw (119bp,713bp) node {15};
\end{scope}
  % Node: 1
\begin{scope}
  \pgfsetstrokecolor{black}
  \draw (440bp,543bp) ellipse (26bp and 26bp);
  \draw (440bp,543bp) node {1};
\end{scope}
  % Node: 3
\begin{scope}
  \pgfsetstrokecolor{black}
  \draw (542bp,717bp) ellipse (26bp and 26bp);
  \draw (542bp,717bp) node {3};
\end{scope}
  % Node: 2
\begin{scope}
  \pgfsetstrokecolor{black}
  \draw (639bp,536bp) ellipse (26bp and 26bp);
  \draw (639bp,536bp) node {2};
\end{scope}
  % Node: 5
\begin{scope}
  \pgfsetstrokecolor{black}
  \draw (397bp,755bp) ellipse (26bp and 26bp);
  \draw (397bp,755bp) node {5};
\end{scope}
  % Node: 4
\begin{scope}
  \pgfsetstrokecolor{black}
  \draw (566bp,392bp) ellipse (26bp and 26bp);
  \draw (566bp,392bp) node {4};
\end{scope}
  % Node: 7
\begin{scope}
  \pgfsetstrokecolor{black}
  \draw (181bp,196bp) ellipse (26bp and 26bp);
  \draw (181bp,196bp) node {7};
\end{scope}
  % Node: 6
\begin{scope}
  \pgfsetstrokecolor{black}
  \draw (212bp,840bp) ellipse (26bp and 26bp);
  \draw (212bp,840bp) node {6};
\end{scope}
  % Node: 9
\begin{scope}
  \pgfsetstrokecolor{black}
  \draw (290bp,367bp) ellipse (26bp and 26bp);
  \draw (290bp,367bp) node {9};
\end{scope}
  % Node: 8
\begin{scope}
  \pgfsetstrokecolor{black}
  \draw (69bp,202bp) ellipse (26bp and 26bp);
  \draw (69bp,202bp) node {8};
\end{scope}
  % Node: 14
\begin{scope}
  \pgfsetstrokecolor{black}
  \draw (50bp,411bp) ellipse (32bp and 33bp);
  \draw (50bp,411bp) node {14};
\end{scope}